\documentclass{amsart}

\RequirePackage[OT1]{fontenc}
\RequirePackage{amsthm,amsmath}
\RequirePackage[numbers]{natbib}
\RequirePackage[colorlinks,citecolor=blue,urlcolor=blue]{hyperref}

	\newtheorem{theorem}{Theorem}[section]
\newtheorem{prop}[theorem]{Proposition}
\newtheorem{lemma}[theorem]{Lemma}
\newtheorem{cor}[theorem]{Corollary}

\newtheorem{qn}[theorem]{Question}
\newtheorem{defn}[theorem]{Definition}

\newtheorem{rmk}[theorem]{Remark}

\newtheorem{eg}[theorem]{Example}

\newtheorem{setup}[theorem]{Setup}

\newcommand{\propref}[1]{Proposition~\ref{#1}}

\newcommand{\dirac}{\textrm{Dirac}}

\newcommand{\G}{{\Gamma}}
\newcommand{\s}{{\Sigma}}

\newcommand{\natls}{{\mathbb N}}

\newcommand{\reals}{{\mathbb R}}

\newcommand\FF{{\mathcal F}}
\newcommand\GG{{\mathcal G}}

\newcommand\LL{{L}}
\newcommand\MM{{\mathcal M}}

\newcommand\PP{{\mathcal P}}
\newcommand\QQ{{\mathcal Q}}

\newcommand\SSS{{\mathcal S}}

\newcommand\MF{{\MM\FF}}
\newcommand\PMF{{\PP\kern-2pt\MM\FF}}

\newcommand\PML{{\PP\kern-2pt\MM\LL}}

\newcommand\ep{\epsilon}

\newcommand\Hyp{{\mathbf H}}

\newcommand\Z{{\mathbb Z}}
\newcommand\R{{\mathbb R}}

\newcommand\E{{\mathbb E}}

\newcommand{\sas}{{S\alpha S}}
\newcommand{\lag}{{\Lambda_G}}
\newcommand{\lagg}{{\Lambda_G^{(2)}}}
\newcommand{\laggr}{{\Lambda_G^{(2)}\times \R}}

\newcommand{\fex}{{f_{ex}}}
\newcommand{\fexr}{{f_{exr}}}
\renewcommand{\L}{{\Lambda}}

\newcommand\gesim{\succ}

\newcommand{\mups}{{\mu^{PS}}}
\newcommand{\mubm}{{\mu^{BM}}}
\newcommand{\mubms}{{\mu^{BMS}}}
\newcommand{\mubmr}{{\mu^{BMR}}}

\newcommand{\omrk}{{\Omega_{(r,k)}}}
\newcommand{\omrkm}{{\Omega_{(r-k,k)}}}
\newcommand{\sro}{{\s_r(o)}}
\newcommand{\iid}{{i.i.d.\ }}
\newcommand{\poin}{{Poincar\'e \,}}
\newcommand{\rosin}{{Rosi\'nski \,}}

\newcommand{\bbar}{\overline}

\newcommand\qcg{{QC(\Lambda)}}

\newcommand\convd{\stackrel{d}{\rightarrow}}
\newcommand\convp{\stackrel{p}{\rightarrow}}

\newcommand\eqd{\stackrel{d}{=}}

\numberwithin{equation}{section}
\allowdisplaybreaks
\usepackage{amsfonts,amssymb,verbatim,amsmath,amsthm,latexsym,textcomp,amscd}
\usepackage{latexsym,amsfonts,amssymb,epsfig,verbatim}
\usepackage{amsmath,amsthm,amssymb,latexsym,graphics,textcomp, hyperref}
\usepackage{graphicx}
\usepackage{color}
\usepackage{url}
\usepackage{psfrag}

\begin{document}
	
	\title[Stable Fields and Extremal Cocycle Growth]{Stable Random Fields, Patterson-Sullivan measures and Extremal Cocycle Growth}
	
	\author{Jayadev S. Athreya}
	\address{Jayadev S. Athreya, Department of Mathematics, University of Washington, Box 354350 Seattle WA 98195-4350, USA}
	
	\email{jathreya@uw.edu}
	
	\author{Mahan Mj}
	\address{Mahan Mj, School of Mathematics, Tata Institute of Fundamental Research, 1 Homi Bhabha Road, Mumbai 400005, India}
	
	\email{mahan@math.tifr.res.in}
	\email{mahan.mj@gmail.com}
	
	\author{Parthanil Roy}
	\address{Parthanil Roy, Theoretical Statistics and Mathematics Unit, Indian Statistical Institute, 8th Mile, Mysore Road, RVCE Post, Bangalore 560059, India}
	
	\email{parthanil.roy@gmail.com}
	
	\thanks{J.S.A.~acknowledges the support of NSF CAREER grant DMS 1559860 and NSF grant DMS 2003528. M.M.~is partially supported by a DST J~C~Bose Fellowship, an endowment from the Infosys Foundation, and by the Department of Atomic Energy, Government of India, under project no.12-R\&D-TFR-5.01-0500. P.R.~is partially supported by a DST SwarnaJayanti Fellowship, and a SERB grant MTR/2017/000513.}

	\date{\today}

\begin{abstract}
We study extreme values of group-indexed stable random fields for discrete groups $G$ acting geometrically on spaces $X$ in the following cases: (1) $G$ acts properly discontinuously by isometries on a CAT(-1)  space $X$, (2) $G$ is a lattice in a higher rank Lie group, acting on a symmetric space $X$, and (3) $G$ is the mapping class group of a surface acting on its Teichm\"uller space. The connection between extreme values and the geometric action is mediated by the action of the group $G$ on its limit set equipped with the Patterson-Sullivan measure. Based on motivation from extreme value theory, we introduce an invariant of the action called  extremal cocycle growth which measures the distortion of measures on the boundary in comparison to the movement of points in the space $X$ and  show that its non-vanishing is equivalent to finiteness of the Bowen-Margulis measure for the associated unit tangent bundle $U(X/G)$ provided $X/G$ has non-arithmetic length spectrum. As a consequence, we establish a dichotomy for the  growth-rate of a partial maxima sequence of stationary symmetric $\alpha$-stable ($0 < \alpha < 2$) random fields indexed by groups acting on such spaces. We also establish analogous results for normal subgroups of free groups.
\end{abstract}

	\maketitle

\section{Introduction} Let $G$ be a discrete finitely generated group acting  and properly discontinuously by isometries on a  space $X$ in one of the following situations:
\begin{enumerate}
	\item $G$ acts properly discontinuously by isometries on a CAT(-1)  space $X$,
	\item $G$ is a lattice in a higher Lie group $\mathcal G$, acting on its symmetric space  $X$.
	\item $G$ is the mapping class group of a surface acting on its Teichm\"uller space.
\end{enumerate}

Let  $\Lambda_G \subset \partial X$ denote the limit set--the collection of accumulation points of an(y) orbit on the boundary $\partial X$.


The aim of this paper is to establish a connection between three perspectives on the action of $G$ on $\Lambda_G$ pertaining to three different themes  as mentioned below:
\begin{enumerate}
	\item maxima of stationary \emph{symmetric $\alpha$-stable} ($\sas$) random fields indexed by $G$  ({\bf Probability Theory}),
	\item extreme values of cocycles given by Radon-Nikodym derivatives of Patterson-Sullivan measures  induced by the quasi-invariant action of $G$ on its limit set $\Lambda_G \subset \partial X$  ({\bf Ergodic  Theory}),
	\item extrinsic geometry of the orbit of $G$ on $X$ in terms of whether the Bowen-Margulis measure is finite or not   ({\bf Non-positively curved and Hyperbolic Geometry}).
\end{enumerate}
The relation between (1) and (2) has been studied in probability in the context of abelian $G$ and free $G$. The relation between  (2) and (3) on the other hand has been studied thoroughly in the context of pairs $(X,G)$ as above. However the connection between (1) and (3) is unexplored territory for pairs $(X,G)$ as above.  We achieve this connection in the present paper via  the mediation of ergodic theoretic techniques (2), which play a key role in the proofs of our main results. One of the main tools we use from ergodic theory is mixing of the geodesic flow with respect to the Bowen-Margulis measure. The basic test case where $G$ is a free group and $X$ its Cayley graph with respect to a standard generating set had been dealt with in \cite{sarkar:roy:2016}; however this example is somewhat orthogonal to the main thrust of the present paper and examples explored therein, as geodesic flow is {\it not} mixing in the case of the free group. To address this largely excluded case of the free group, we devote a final subsection to  normal subgroups of free (or more generally hyperbolic) groups, where the Bowen-Margulis measure is used and we recover the corresponding theorem from \cite{sarkar:roy:2016}.

The connection between the probabilistic and the ergodic theoretic perspectives ((1) and (2) in the above list) is, in the general form that constitutes the background of this paper, due essentially to Rosi{\'n}ski \cite{Rosinski:1994,Rosinski:1995,Rosinski:2000} (see also the encyclopedic monograph \cite{samorodnitsky:taqqu:1994} and the recent survey \cite{roy:2017}). The study of stationary \emph{$\sas$ random fields} (i.e., stochastic processes indexed by $G$ such that each finite linear combination follows an $\sas$ distribution) is important in probability theory because such fields appear as scaling limits of regularly varying random fields having various dependence structures. These random fields come naturally equipped with a Rosi{\'n}ski representation, thus connecting with measurable dynamical systems in a canonical manner. The naturality of \emph{$\sas$ random fields} in the context of dynamical/ergodic-theoretic applications is in fact a consequence of  the exact correspondence, furnished by the Rosi{\'n}ski representation, between such stochastic processes and \emph{quasi-invariant} (or \emph{nonsingular}) group actions, and hence dynamical cocycles. We outline the connection in Section \ref{sec:sas} and summarize the discussion as follows (for details, see Theorem \ref{omni-sas}).

Given a standard measure space $(S,\mu)$ equipped with a \emph{quasi-invariant} (i.e., measure-class preserving) group action $\{\phi_g\}_{g \in G}$, a $\pm 1$-valued cocycle $\{c_g\}_{g\in G}$ (that is, $$c_{gh}(s) = c_h(s) + c_g(\phi_h(s))$$ for $\{\phi_g\}$ and a function $f \in \LL^\alpha (S,\mu)$, there exists a stationary $\sas$ random field $\{Y_g\}$ indexed by $G$ admitting an integral representation (known as the Rosi{\'n}ski representation):
\begin{equation} \label{eqn:Ros_Repn}
Y_g \eqd \int_S c_g(x)\left(\frac{d(\mu\circ\phi_g)}{d\mu}(x)\right)^{1/\alpha} f\circ \phi_g(x) dM(x), \mbox{\ \ } g \in G,
\end{equation}
where the above integral is with respect to an $\sas$ random measure $M$ on $S$ with control measure $\mu$. We recall that a random measure $M$ is called an $\sas$ random measure with control measure $\mu$ if for each set $A$ with $\mu(A) < \infty$,  the random variable $M(A)$ follows an $\sas$ distribution with scale parameter $(\mu(A))^{1/\alpha}$; see, for example, \cite{samorodnitsky:taqqu:1994}. 

Conversely, given a stationary $\sas$ random field $\{Y_g\}$ indexed by $G$, there exist a standard measure space $(S,\mu)$ equipped with a quasi-invariant group action $\{\phi_g\}_{g \in G}$, a $\pm 1$-valued cocycle $\{c_g\}_{g\in G}$ and a function $f \in \LL^\alpha (S,\mu)$ such that $Y_g$ admits a Rosi{\'n}ski representation given as above.\\

When $\mu$ is a probability measure (often the case in this paper), we shall use $\Lambda$ to denote the space $S$ (as our probability measures will be typically supported on limit sets $\Lambda$). With this change of notation, the basic probabilistic question we address in this paper is:
\begin{qn}\label{mainq}
	Find sufficient conditions on a non-singular conservative action of $G$ on a probability measure space $(\Lambda, \SSS, \mu)$
	to ensure that the growth of partial maxima of the associated  stationary $\sas$ random field indexed by $G$ is like the i.i.d.\ case.
\end{qn}
When $G= \Z^d$ and $X$ is a Cayley graph of $G$ with respect to a standard generating set,  this can never happen \cite{samorodnitsky:2004a, roy:samorodnitsky:2008}. There is only one recent example giving a positive answer to Question \ref{mainq}:  $G=F_d$ is  free, $X$ is a Cayley graph of $G$ with respect to a standard generating set,  and $\L$ is the Cantor-set boundary of $F_d$ equipped with the Patterson-Sullivan measure \cite{sarkar:roy:2016}.
In this paper we prove that there is a large class of examples, {\it geometric} in origin, giving a positive answer to Question \ref{mainq}
(see Theorem \ref{main-nvecg}):

\begin{enumerate}
	\item Non-elementary Gromov-hyperbolic groups $G$ acting on a Cayley graph $X=\Gamma_G$ (with respect to a finite generating set) and on the boundary $\Lambda = \partial G$, equipped with the Patterson-Sullivan measure class. This directly generalizes the main theorem of \cite{sarkar:roy:2016}.
	\item Groups $G$  acting  on proper CAT(-1) spaces $X$ with limit set $\Lambda \subset \partial X$ equipped with the Patterson-Sullivan measure class,  and satisfying the following:
	\begin{itemize}
		\item The length spectrum of $X/G$ is non-arithmetic,
		\item the associated Bowen-Margulis measure on the  unit tangent bundle $UM$ of $M = X/G$ is finite.
	\end{itemize}
	\item A lattice $G$ in a higher rank lie group, acting on the symmetric space $X$ and its Furstenberg boundary $\Lambda= \partial X$ equipped with the Patterson-Sullivan measure class.
	\item The mapping class group $G$ acting geometrically on Teichm\"uller space $X$ and measurably on the Thurston boundary $\Lambda = \partial X$ equipped with the Thurston measure.
\end{enumerate}

For  groups $G$ acting  on proper CAT(-1) spaces $X$ with limit set $\Lambda \subset \partial X$ (as in Item (2) above),
finiteness of the Bowen-Margulis measure $\mubm$ of $UM$ in fact provides a new  phase transition boundary for the behavior of the growth of partial maxima. The growth of partial maxima is like the i.i.d.\ case if and only if the Bowen-Margulis measure $\mubm$ of $UM$  is finite (Theorem   \ref{main-nvecg} Item(2) and Theorem  \ref{ecg-gi}).  An important technical tool  that we use in the proofs of the main Theorems \ref{main-nvecg} and Theorem  \ref{ecg-gi} is mixing of the geodesic flow (in cases (2), (3), (4) above).
Mixing of the geodesic
flow in turn is used to count the number of orbit points inside an $n-$ball.

Mixing, in this strong form, fails for Gromov hyperbolic groups equipped with the word metric \cite{bader-furman}. However, for infinite normal subgroups of infinite index in such groups, we establish  a slightly weaker  counting technique for the number of orbit points. This allows us to obtain Theorem \ref{ecg-hypgp-normal}: the behavior of partial maxima for a normal subgroup $H$ of a hyperbolic group $G$ is \iid-like if and only if $H$ is of finite index in $G$. In the setup of hyperbolic groups, the latter provides the analog of Theorem  \ref{ecg-gi} -- the  $\mubm (UM) = \infty$ case  for CAT(-1) spaces.

A key aim in this paper is to bring into focus the geometry underpinning Question \ref{mainq}. We   replace the default word metric of earlier works on the subject \cite{samorodnitsky:2004a,roy:samorodnitsky:2008,sarkar:roy:2016} by a general proper geodesic metric space $(X,d)$. Apart from Roblin's fundamental dichotomy on the behavior of the Poincar\'e series \cite{roblin-memo},
the tools we bring in to answer Question \ref{mainq} are also from the more geometric aspects of ergodic theory:  mixing of the geodesic flow and  equidistribution of spheres. We introduce  an invariant called {\bf extremal cocycle growth} incorporating both the geometry of the action of $G$
on $(X,d)$ as well as the quasi-invariant action of $G$ on $(\L, \SSS, \mu)$ whose asymptotic qualitative behavior determines the answer to  Question \ref{mainq}. This invariant records the appropriately normalized maximal distortion of the measure $\mu$ at a point $\xi \in \L$ with respect to actions of group elements $g$ which move a fixed point $o \in X$ a bounded amount.

\subsection{Densities and Extremal Cocycle Growth}\label{sec-ecg}
\subsubsection{Horofunction boundary and Busemann function:}
Throughout  $(X,d)$ will be a proper geodesic metric space.
\begin{defn}
	\cite{gromov-hypm}
	Let $\hat{C}(X) = C(X)/\sim$ denote  the compact set of $1$-Lipschitz functions with the topology of convergence on compact subsets, where $f \sim g$ if $f-g$ is a constant.
	Embed $X$ in $\hat{C}(X)$ via $i: x \to d_X(x,\cdot)$. The closure $\bbar{i(X)}$ is called the {\bf horofunction-compactification} of $X$ and $\partial_h X = \bbar{ i(X)}\setminus i(X)$ is the  {\bf horofunction-boundary} of $X$.
	
	Given a (parameterized) geodesic ray
	$\gamma \subset (X,d)$,
	the {\bf Busemann function} based at $x$ for $\gamma$ is given by
	$$\beta_\gamma(x)=\lim_{t \to \infty} \left(d(x,  \gamma(t))-t\right).$$
\end{defn}

If $X$ is CAT(0), $\bbar{i(X)}=\overline{X}$ equals the usual {\bf visual compactification} $\bbar X$ and $\partial_h X$ is the visual boundary. If $X=Teich(S)$ is the Teichm\"uller space, $\bbar i(X)$ gives the {\bf Gardiner-Masur compactification} (see \cite{miyachi-hb} and references therein).

If $(X,d)$ is Gromov-hyperbolic and
$\xi$ is an endpoint of a geodesic $\gamma$, the equivalence class of
$\beta_\gamma$'s with $\gamma(\infty) = \xi$ will be denoted as $\beta_\xi$.
Also, if $(X,d)$ is Gromov-hyperbolic, quotienting $\partial_h X$ further by bounded functions we obtain the  {\bf Gromov boundary} $\partial_g {X}$ \cite[Section 2.5]{calegari-notes}.  The pre-image of $\xi$ (under this further projection) are the elements of the equivalence class $\beta_\xi$. To get a well-defined Busemann function in this case, instead of an equivalence class, we shall define
\begin{equation}\label{eq:buse}
\beta_\xi (p,q) :=\limsup_{z\to \xi}\ (d(p,z)-d(q,z)).
\end{equation}

In all three cases (CAT(0),  Teichm\"uller, or Gromov-hyperbolic) we shall choose a base-point $o$ and normalize Busemann functions such that
$\beta_\xi (o)=0$. If in addition $X$ is a Cayley graph of a  hyperbolic group, $o$ will be the identity.

\subsubsection{Quasiconformal density:}
Now suppose $G$ acts properly discontinuously by isometries on $X$.
\begin{defn}\cite[p. 721]{cm-gafa} \label{qcdens} Let $M(\partial_hX) $ denote the collection  of positive
	finite Borel measures on  $\partial_hX $.
	A
	$G-$invariant
	{\bf conformal density}
	of dimension $v$ ($v\geq 0$)
	on $\partial_hX$
	is a
	continuous
	$G-$equivariant
	map $X \to M(\partial_hX)$
	sending
	$x \to \mu_x$
	such that
	$$\frac{d\mu_x}{d\mu_o}(\xi)=exp(-v\beta_\xi(o,x)).$$
	
	For $X$  Gromov-hyperbolic, let $M(\partial_g X)$ denote the collection  of positive
	finite Borel measures on  $\partial_g X $. A  $G-$equivariant map $X \to M(\partial_g X)$ sending $x \to \mu_x$
	is said to be a $C-$quasiconformal
	density
	of dimension
	$v$ ($v\geq 0$),
	for some $C\geq 1$, if
	\begin{equation}\label{eq-qc}
	\frac{1}{C}exp(-v\beta_\xi(y,x))
	\leq
	\frac{d\mu_{x}}{d\mu_y} (\xi)
	\leq C exp(-v\beta_\xi(y,x))
	\end{equation}
	for all $x, y \in X$, $\xi\in \partial X$; in particular,
	$$\frac{1}{C} exp(-v\beta_\xi(o,g.o))\leq \frac{d\mu_{g.o}}{d\mu_o} (\xi)
	\leq C exp(-v\beta_\xi(o,g.o)),$$ for all $g \in G$.
\end{defn}	

If Equation \ref{eq-qc} holds for some $C>0$ and all $x, y \in X$,  $\xi\in \partial X$ we shall simply write
$$\frac{d\mu_{x}}{d\mu_y} (\xi)
\asymp  exp(-v\beta_\xi(y,x)), $$ omitting the specific value of $C$.
\begin{setup}\label{setup} The setup for the rest of the paper is as follows. $(G,X,\L)$ will denote one of the following:
	\begin{enumerate}
		\item $X$ is a proper  CAT(-1) space,  $G$  a non-elementary discrete group acting properly discontinuously by isometries on $X$ (recall that $G$ is non-elementary means that its limit set is infinite), and $\L$ the limit set.
		\item  $X$ is a proper Gromov-hyperbolic space,   $G$  a non-elementary discrete group acting properly discontinuously by isometries on $X$, and $\L$ the limit set.
		\item $X$ is the Teichm\"uller space $Teich(S)$ of a closed surface $S$,  $G=MCG(S)$ acting on $Teich(S)$, and $\L$ the Thurston Boundary $\PMF(S)$.
		\item $X$ is a symmetric space for a Lie group $\mathcal G$ of higher rank,   $G$ is a lattice in $\mathcal G$, and $\L$ the Furstenberg boundary of $X$.
	\end{enumerate}
	
	We shall  use the convention that $\partial X$ stands for
	the horofunction boundary $\partial_h X$ in cases 1, 3, 4.
	and the Gromov boundary $\partial_g X$ in Case 2.
	
	In all the cases $\mu$ will denote a conformal or quasiconformal density on $\L \subset \partial X$ (see Section \ref{sec:ps} for  existence).
\end{setup}

\subsubsection{Extremal Cocycle Growth (ECG):}\label{subsec-ecg} Let $(X, G, \L)$ be as in Setup \ref{setup}.  Let $x \to \mu_x$ be a $G$-invariant conformal or quasiconformal density of dimension $v$. Define
\begin{equation}\label{eq-vn}
V_n = exp(vn).
\end{equation}
Let $B(o,n)$ denote the $n-$ball about $o \in X$ and $$B_n =\{g \in G\vert g.o \in B(o,n)\}.$$ Since $\mu_x, \mu_y$ are absolutely continuous with respect to each other by Equation \ref{eq-qc}, they have the same support. Let $\L$ denote this support. In the cases that we shall be interested in, $\mu_o(\L) $ is finite and hence without loss of generality, we can assume that $\mu_o$ is a probability measure. For  $\xi \in (\L,\mu)$ we evaluate the point-wise partial maxima of the Radon-Nikodym derivatives and let
\begin{equation}\label{eq-anxi}
A_n(\xi) := \max_{g \in B_n} \left[\frac{d\mu_{g.o} }{d \mu_o}(\xi) \right].
\end{equation}
To simplify notation, we let $\mu_o=\mu$, so that $\mu_{g.o}=g_\ast \mu_o$, i.e.\ $\mu_{g.o}(A) = \mu_o (g^{-1}(A))$ and $$A_n(\xi) := \max_{g \in B_n} \left[\frac{dg_\ast\mu}{d \mu}(\xi) \right].$$
This $A_n(\xi)$ is the maximal distortion of the reference measure $\mu$ at the point $\xi$ under group elements $g$ which move the basepoint $o$ at most distance $n$.
Let
\begin{equation}\label{eq-an}
\bbar{A_n} := \int_\L A_n(\xi) d\mu(\xi)
\end{equation}
be the {\bf expectation} of the  point-wise partial maxima $A_n(\xi)$. We shall call $A_n$ the extremal value of the cocycle $\frac{dg_\ast\mu }{d \mu}$. Finally define the {\bf normalized extremal cocycle}
\begin{equation}\label{eq-cn}
C_n:= \frac{\bbar{A_n}}{V_n}.
\end{equation}  We shall refer to
the asymptotics of $C_n$ (as $n\to \infty$) as {\bf extremal cocycle growth}. More precisely,

\begin{defn}\label{def-ecgps}
	We shall say that the action of $G$ on $(\L,\mu)$
	\begin{enumerate}
		\item has {\bf vanishing extremal cocycle growth} (vanishing ECG for short) if $$\, \lim_{n\to\infty} C_n=0,$$ i.e.\ the limit exists and equals zero;
		\item  has {\bf non-vanishing extremal cocycle growth}
		(non-vanishing ECG for short) if $$\, \lim\inf_{n\to\infty}C_n>0.$$
	\end{enumerate}
\end{defn}

It follows from Definition~\ref{qcdens}  that $$\bbar{A_n} \asymp \int_\L \max{g \in B_n} \left(\exp(v \beta_{\xi}(o, g.o))\right) d\mu(\xi)$$
\begin{equation}\label{eq:ecg}
C_n \asymp \frac{1}{V_n} \int_\L \max_{g \in B_n}\left( \exp(v \,  \beta_{\xi} \, (o, g.o))\right)
d\mu(\xi)
\end{equation}

In this paper  we shall be interested in the dichotomy given by zero and non-zero extremal  cocycle growth. It will suffice therefore to estimate the asymptotics of the RHS of Equation~\eqref{eq:ecg}.

\subsubsection{A brief example}\label{sec-sl2z} A common inspiration for many of the settings we study in this paper is the action of the group $SL(2,\mathbb Z)$ on the hyperbolic plane $X = \mathbb H^2$. The limit set of $SL(2, \mathbb Z)$ is $\R \cup \infty,$ and the action on the boundary (and on $\mathbb H^2$) is given by fractional linear maps, $$\phi_g(\xi) = \frac{a\xi+b}{c\xi+d},$$ where $g = \begin{pmatrix} a & b \\ c& d \end{pmatrix}.$ The Lebesgue measure class is preserved, and if we choose standard Lebesgue measure as our reference measure $\mu$,  the Radon-Nikodym derivative $$\frac{d g_*\mu}{d\mu}(\xi) = \frac{1}{(c\xi+d)^2}.$$ If we take, for example our basepoint $o$ in $X = \mathbb H^2$ as $i$, and our radius $n = \log 3,$ there are $5$ elements of $SL(2, \Z)$ (we are eliding the issue of elements that stabilize $i$, here) which have $$d(g.o, 0) \le n,$$ namely the identity, and the matrices $$ \begin{pmatrix} 1& \pm 1 \\ 0 &1\end{pmatrix}, \begin{pmatrix} 1& 0 \\ \pm 1 & 1\end{pmatrix}.$$ If we take $\xi = 2$ as our reference point on the boundary, we have that the Radon-Nikodym derivatives $$\frac{1}{(2c+d)^2}$$ which takes on the values $1$ and $1/3$ at the matrices above, so $A_{\log 3}(2) = 1,$ in this formulation. The group $SL(2, \Z)$ can also be viewed as acting on its Cayley graph, which gives a different interpretation which is also generalized in our work.

\subsubsection{Free groups} We point out here that the free group on 2 generators $F_2$ provides us with examples to which both Theorem   \ref{main-nvecg}  and Theorem  \ref{ecg-gi} apply. The group $G=F_2$ thus furnishes three kinds of examples of non-singular conservative actions.  
\begin{enumerate}
	\item $F_2$ acting on its own boundary equipped with the Patterson-Sullivan measure. Here, extremal cocycle growth is non-vanishing and hence the growth of partial maxima of the associated $\sas$ random field is like the \iid case.
	\item  $F_2$ may be identified with an index 6 subgroup of $PSL(2, \mathbb Z)$
	(the second congruence subgroup). Hence, as in Section \ref{sec-sl2z}, it acts on the circle $\R \cup \infty,$ equipped with the Lebesgue measure. This corresponds to an action where the associated Bowen-Margulis measure 
	of the unit tangent bundle of $\Hyp^2/F_2$ is finite. Again, 
	extremal cocycle growth is non-vanishing (by Theorem   \ref{main-nvecg}) and hence the growth of partial maxima of the associated $\sas$ random field is like the \iid case.
	\item $F_2$ arises as a normal	subgroup of the fundamental group of the figure eight knot complement $M$. The 3-manifold $M$ admits a hyperbolic structure \cite{thurstonnotes}. Hence $F_2$ acts on $\Hyp^3$ freely, properly discontinuously by isometries. The limit set in this case turns out to be the
	whole boundary $\partial \Hyp^3 = S^2$. The sphere $S^2$ is again equipped with the Lebesgue measure class, which is preserved by the $F_2-$action.
	It turns out that extremal cocycle growth is vanishing (Theorem  \ref{ecg-gi}) and the growth of partial maxima of the associated $\sas$ random field is {\it not} like the \iid case.
\end{enumerate}

The above examples illustrate that the growth of partial maxima of the associated $\sas$ random field does not depend on the group $G$ alone but rather on the geometry of the space $X$ on which it acts, and via this action, on the specific nature of the associated probability measure space
$(\Lambda, \SSS, \mu)$ on which $G$ admits a non-singular action.

\subsubsection{Outline of the paper} In Section \ref{sec:sas}, we give a brief review of group indexed $\sas$-random fields $\{X_g\}_{g\in G}$ and deduce a basic criterion (Theorem \ref{main_prob}) in terms of  non-vanishing or vanishing of ECG (Definition \ref{def-ecgps}) that determines whether  the partial maxima of  $\{X_g\}_{g\in G}$ exhibits \iid-like behavior or not. This reduces the purely probabilistic question \ref{mainq} to the following  question lying at the interface of geometry, dynamics and probability:

\begin{qn}\label{mainq-dyna}
	Find sufficient conditions on triples $(G,X,\L)$ such that ECG is non-vanishing.
\end{qn}

In Section \ref{sec:psbm}, we recall various theorems from the literature that show that Patterson-Sullivan measures in the context of Setup \ref{setup} give quasiconformal densities. We also recall work of Furman~\cite{furman-coarse} and Bader-Furman~\cite{bader-furman} on Bowen-Margulis measures.
In Section \ref{sec:mix} we recall results on mixing of the geodesic flow and establish consequences on convergence of spherical averages. In the special case of the mapping class group acting on Teichm\"uller space, the corresponding result
(Theorem \ref{teichmix}) appears here for the first time.
Section \ref{sec:ecg-sa} is the technical core of the paper and relates spherical averages to ECG. In Section \ref{sec-mainthm} we prove the main Theorems of the paper. Theorem \ref{main-nvecg} establishes non-vanishing of ECG in the four cases mentioned at the beginning of the Introduction and Theorem \ref{ecg-gi} establishes vanishing of ECG for CAT(-1) examples with infinite Bowen-Margulis measure. Normal subgroups of hyperbolic (e.g.\ free) groups are treated in Section \ref{sec:ecgnormal}.

\section{Group-indexed stable random fields}\label{sec:sas}
We shall use  $(S, \SSS, \mu)$ to denote a $\sigma$-finite general Borel measure space and $(\L, \SSS, \mu)$  to denote a probability measure space.
\begin{defn}
	A (real-valued) random variable $Y$ is said to follow a {\bf symmetric $\alpha$-stable} ($\sas$) distribution with tail parameter $\alpha \in (0, 2]$ and scale parameter $\sigma > 0$ if it
	has characteristic function of the form $\mathbb{E}(e^{i\theta Y}) = \exp{\{-\sigma^{\alpha}|\theta|^\alpha\}}$, $\theta \in \mathbb{R}$.
\end{defn}
The value of the tail parameter equal to 2 corresponds to the Gaussian case.  Here we shall largely focus on $\alpha \in (0,2)$, i.e.\
the non-Gaussian case (see \cite{samorodnitsky:taqqu:1994} for a detailed treatment of $\alpha$-stable ($0<\alpha <2$) distributions).

\begin{defn}
	Let $G$ be a finitely generated infinite group with identity element $e$.
	A random field (that is, a collection of random variables) $\mathbf{Y}=\{Y_g\}_{g\in G}$ indexed by $G$ is called an $\sas$ random field if for all $k \geq 1$,
	$g_1, g_2, \ldots, g_k \in G$ and  $c_1, c_2, \ldots, c_k \in \mathbb{R}$, the linear combination $\sum_{i=1}^k c_i Y_{g_i}$  follows an $\sas$ distribution.
\end{defn}

\noindent {\bf Integral Representations:}
Any such random field has an integral representation of the type
\begin{equation}\label{integrep}
Y_g \eqd \int_S f_g(x)M(dx), \mbox{ \ \  } g \in G,
\end{equation}
where $M$ is an S$\alpha$S random measure on some $\sigma$-finite standard Borel space $(S,\mathcal{S}, \mu)$, and $f_g \in \mathcal{L}^\alpha (S, \mu)$ for all $g \in G$; see
Theorem 13.1.2 of \cite{samorodnitsky:taqqu:1994}. This simply means that each linear combination $\sum_{i=1}^k c_i Y_{g_i}$  follows an S$\alpha$S distribution
with scale parameter $\|\sum_{i=1}^k c_i f_{g_i}\|_\alpha$. We shall always assume, without loss of generality, that
$$\bigcup_{g \in G} \{x \in S: f_g(x) \neq 0\} = S$$
modulo $\mu$. That is, for $\mu$-a.e. $x \in S$, there is a $g \in G$ so that $f_g(x) \neq 0$.

\begin{defn}
	The field $\{Y_g\}_{g \in G}$ is called {\bf left-stationary}
	if $\{Y_g\} \overset{d}{=} \{Y_{hg}\} $ for all $h\in G$, i.e.\ the joint distributions of the $k-$tuples $(Y_{g_1},Y_{g_2}, \cdots, Y_{g_k} )$
	and $(Y_{hg_1},Y_{hg_2}, \cdots, Y_{hg_k} )$ are equal for all $k$ and all $k-$tuples $({g_1},{g_2}, \cdots, {g_k} )$.
\end{defn}
We shall simply write stationary to mean left-stationary throughout this paper.

\begin{theorem} \label{omni-sas} {\bf (Rosi{\'n}ski Representation) \cite{Rosinski:1994,Rosinski:1995,Rosinski:2000}}
	Given a standard measure space $(S,\mu)$ equipped with a quasi-invariant group action $\{\phi_g\}_{g \in G}$, a $\pm 1$-valued cocycle $\{c_g\}_{g\in G}$ for $\{\phi_g\}$,  and an $f \in \LL^\alpha (S,\mu)$, there exists a stationary $\sas$ random field indexed by $G$  admitting an integral representation (known as the Rosi{\'n}ski representation):
	\begin{equation}\label{eqthm}
	f_g(x)=c_g(x)\left(\frac{d(\mu\circ\phi_g)}{d\mu}(x)\right)^{1/\alpha}\left( f\circ \phi_g\right)(x), \mbox{\ \ } g \in G.
	\end{equation}
	Conversely, given a stationary $\sas$ random field $\{Y_g\}$ indexed by $G$, there exist a standard measure space $(S,\mu)$ equipped with a quasi-invariant group action $\{\phi_g\}_{g \in G}$, a $\pm 1$-valued cocycle $\{c_g\}_{g\in G}$ and an $f \in \LL^\alpha (S,\mu)$ such that $Y_g$ admits a Rosi{\'n}ski representation given by \eqref{eqthm}.
\end{theorem}

The stationary random field indexed by $G$ corresponding to the standard measure space $(S,\mu)$, the quasi-invariant group action $\{\phi_g\}_{g \in G}$, the $\pm 1$-valued cocycle $\{c_g\}_{g\in G}$  and an $f \in \LL^\alpha (S,\mu)$ is denoted as
$$Y_g:=Y_g (S,\mu, \{\phi_g\},  \{c_g\}, f), g \in G.$$
In the special case that $f \equiv 1$ (here we must have $\mu(S) < \infty$), we simplify the notation to $$Y_g:=Y_g (S,\mu, \{\phi_g\},  \{c_g\}, 1)=Y_g (S,\mu, \{\phi_g\},  \{c_g\}), g \in G.$$ If  $(S, \mu)$ is a probability measure space, we shall replace $S$ by $\L$.

In this work, we are interested in the rate of growth of the partial maxima sequence
$$M_n = \max_{g\in B_n}  |Y_g|$$
as $n$ increases to $\infty$. It was shown  in \cite{samorodnitsky:2004a, roy:samorodnitsky:2008} that when $G=\mathbb{Z}^d$, the rate of growth of $M_n$ is like the i.i.d. case if and only if the action of $\mathbb{Z}^d$ on  $(S,\mu)$ in Theorem \ref{omni-sas} above is not conservative. In general, the rate of growth of $M_n$ is controlled by that of the deterministic sequence
$$
b_n  =\left( \int_S \max_{g \in B_n} |f_g(x)|^\alpha \mu(dx) \right)^{1/\alpha}.
$$
See, for example, Section 3 of \cite{samorodnitsky:2004a}. The following proposition relates this sequence with the extremal value of the cocycle defined in \eqref{eq-an}.

\begin{prop}\label{pmax=an}
	Let $(X,G,\partial X)$ be as in Setup \ref{setup}. Let $$Y_g:=Y_g (\L,\mu, \{\phi_g\},\{c_g\}), g \in G,$$ be a stationary $\sas$ random field indexed by $G$ where $\mu$ is a quasiconformal measure supported on $\L \subset \partial X$. Then $b_n = \bbar{A_n}^{1/\alpha}$, where $A_n$ is as in \eqref{eq-an}.
\end{prop}

\begin{proof} Incorporating the form of Rosi{\'n}ski representation in $b_n$, we get
	\begin{align}
	b_n & =\left( \int_\L \max_{g \in B_n} |f_g(x)|^\alpha \mu(dx) \right)^{1/\alpha} \nonumber \\
	& = \left( \int_\L\max_{g \in B_n} \left[|f \circ \phi_g(x)|^\alpha \frac{d\mu \circ \phi_g}{d \mu}(x) \right]\mu(dx) \right)^{1/\alpha}. \nonumber 
	\end{align}
	Therefore, for $Y_g:=Y_g (\L, \mu, \{\phi_g\}), g \in G$,
	\begin{equation}\label{eqn:bn}
	b_n = \left( \int_\L \max_{g \in B_n} \left[\frac{d\mu \circ \phi_g}{d \mu}(x) \right]\mu(dx) \right)^{1/\alpha}.
	\end{equation}
	Hence from Equation \ref{eq-an}, $b_n = \bbar{A_n}^{1/\alpha}$.
\end{proof}

\subsection*{A sufficient condition for \iid-like behavior}\label{crit}
The purpose of the rest of this section is to show that if the action of $G$ on $(\L,\mu)$ has non-vanishing extremal cocycle growth (Definition \ref{def-ecgps}), then the growth of partial maxima of the associated $G-$indexed $\sas$ random field (via the Rosi\'nski representation Theorem \ref{omni-sas}) behaves like an \iid random field (see Definition \ref{def-iidlike} and Theorem \ref{main_prob} below for a precise statement). This may be done by recasting the proof of Theorem~4.1 of \cite{samorodnitsky:2004a} in our setup. We  present  a  sketch below along with the relevant modifications.

\begin{theorem} \label{main_prob}
	Consider a $G$-indexed stationary $\sas$ random field $\{Y_g\}$ with $$Y_g:=Y_g (\L,\mu, \{\phi_g\},  \{c_g\}, 1)=Y_g (\L,\mu, \{\phi_g\},  \{c_g\}), g \in G,$$ as in Proposition \ref{pmax=an}. Let $$M_n = \max_{g\in B_n}  |Y_g|.$$  Then the following dichotomy holds:
	\begin{enumerate}
		\item If the action of $G$ on $(\L,\mu)$ has non-vanishing extremal cocycle growth, then given any subsequence of $\{M_n\}$, there exists a further subsequence  $\{M_{n_k}\}$ such that
		\begin{equation}
			\frac{M_{n_k} }{V_{n_k}^{1/\alpha}} \convd \kappa Z_\alpha, \ {\rm as} \ n \to \infty, \label{convdist_to_Fr}
		\end{equation}
		where $\convd$ denotes convergence in distribution. Further, $Z_\alpha$ is a Frech\'{e}t type extreme value random variable and $\kappa$ is a positive constant that may depend on the choice of the subsequence $\{M_{n_k}\}$. If further $lim_{n \to \infty} C_n$ exists (and hence is positive), then $M_{n} / V_{n}^{1/\alpha}$ converges weakly to the limit in \eqref{convdist_to_Fr}.

		\item If the action of $G$ on $(\L,\mu)$ has vanishing extremal cocycle growth, then
		\begin{equation}
			\frac{M_n}{V_n^{1/\alpha}} \convp 0, \ {\rm as} \ n \to \infty, \label{convprob_to_zero}
		\end{equation}
		where $\convp$ denotes convergence in probability.
	\end{enumerate}
\end{theorem}

\begin{defn} \label{def-iidlike} If the sequence of partial maxima of a group indexed stationary random field $\{Y_g\}$ satisfies Equation \ref{convdist_to_Fr}, we say that $\{Y_g\}$ is {\bf \iid-like} with respect to the behavior of partial maxima.
\end{defn}

The key ingredient of the proof of Theorem~\ref{main_prob} is the following series representation (see, for instance, Equation (4.12) of \cite{samorodnitsky:2004a}): For all fixed $n \geq 1$,
\begin{equation}
	(Y_g)_{g \in B_n} \eqd \left(b_n \mathfrak{C}_\alpha^{1/\alpha} \sum_{j=1}^{\infty} \varepsilon_j \Gamma_j^{-1/\alpha} \frac{f_g(U^{(n)}_j)}{\max_{h \in B_n}|f_h(U^{(n)}_j)|}\right)_{g \in B_n}, \label{series_repn_1}
\end{equation}
where
\begin{enumerate}
	\item
	``$\eqd$'' denotes equality of distribution,
	\item $b_n$'s are given by Equation \eqref{eqn:bn},
	\item \begin{equation*}
		\mathfrak{C}_\alpha =\left(\int_0^\infty x^{-\alpha}\sin x dx \right)^{-1}
		= \left\{
		\begin{array}{ll}
			\frac{1-\alpha}{\Gamma(2-\alpha)\cos(\pi \alpha/2)} & \mbox{if $\alpha \ne 1, $}\\
			\frac{2}{\pi}                                        & \mbox{if $\alpha =1$,}
		\end{array}
		\right.
	\end{equation*}
	\item $\varepsilon_j$'s are i.i.d. Bernoulli random variables taking $\pm 1$ values with equal probability,
	\item  $\{U^{(n)}_j: j \geq 1\}$ is an i.i.d. sequence of $\L$-valued random variables with common law given by
	\[
	\mathbb{P}(U^{(n)}_1 \in W)  = b_n^{-\alpha} \int_{W} \max_{h \in B_n}  |f_h(x)|^\alpha \mu(dx), \mbox{ and}
	\]
	\item for all $j \geq 1$, $\Gamma_j = E_1 + E_2 + \cdots E_j$ with $E_j$'s being i.i.d.\ exponential random variables with unit mean.
\end{enumerate}
Note that the right hand side of \eqref{series_repn_1} converges almost surely, and the equality of distribution can be verified for each linear combination of the two sides with the help of Theorem~1.4.2 of \cite{samorodnitsky:taqqu:1994}.

\begin{proof}[Sketch of Proof of Theorem~\ref{main_prob}] We split into two cases:
	
	\medskip
	
	\noindent {\bf Case 1: ECG is non-vanishing:}
	For simplicity, let us assume that $\lim_{n \to \infty} C_n$ exists (not just the limit inferior) and hence is positive. With this assumption, $b_n/V_n^{1/\alpha} = (A_n/V_n)^{1/\alpha}$ converges to a positive constant and as in the proof of Equation (4.9) in \cite{samorodnitsky:2004a} (see for instance the heuristics below), it follows that
	\begin{equation}
		\frac{M_n}{b_n} \convd c Z_\alpha \label{main_step_for_heu}
	\end{equation}
	for some $c > 0$. This completes a (sketch of a) proof of the last statement in Theorem~\ref{main_prob}.
	
	In the general case,  $\liminf_{n \to \infty} C_n$ exists and is positive. Hence given any subsequence of $\{C_n\}$, there is a further subsequence $\{C_{n_k}\}$ that converges to a positive limit. Therefore, applying the argument used above on this subsequence, we obtain \eqref{convdist_to_Fr}.  Now, by Theorem~3.1 of \cite{sarkar:roy:2016}   Case 1 follows. \\
	
	\medskip

	\noindent	{\bf  Case 2: ECG is vanishing:} The proof of \eqref{convprob_to_zero} relies on a comparison argument given in \cite{samorodnitsky:2004a} (see the proof of Equation (4.3) therein). As in Example 5.4 of \cite{samorodnitsky:2004a} and Example 6.1 of  \cite{sarkar:roy:2016}, we construct an auxiliary stationary $\sas$ random field $\{Y^\prime_g\}_{g \in G}$ with Rosi\'{n}ski representation given by
	\[
	Y^\prime_g \eqd \int_{S^\prime} c^\prime_g(x)\left(\frac{d\mu^\prime\circ\phi^\prime_g}{d\mu^\prime}(x)\right)^{1/\alpha} f^\prime \circ \phi^\prime_g(x) M^\prime(dx), \mbox{\ \ } g \in G,
	\]
	on a standard probability space $(\L^\prime, \mu^\prime)$ such that
	\begin{equation}
		a V_n^\epsilon \leq b^\prime_n :=\left( \int_{L^\prime} \max_{g \in B_n} |f^\prime_g(x)|^\alpha \mu^\prime(dx) \right)^{1/\alpha} =o(V_n^{1/\alpha}) \label{condn_on_Xprime}
	\end{equation}
	for some $a >0$ and $\epsilon \in (0, 1/\alpha)$. Without loss of generality, we may assume that $\L$ and $\L^\prime$ are disjoint sets.
	
	We consider the stationary $\sas$ random field
	\[
	Z_g = Y_g + Y^\prime_g, \;\; g \in G,
	\]
	which has a canonical Rosi\'{n}ski representation on $\L \cup \L^\prime$ with the action being $\{\phi_g\}$ restricted to $(\L, \mu)$ and $\{\phi^\prime_g\}$ restricted to $(\L^\prime, \mu^\prime)$. Therefore, the $\{b^Z_n\}$ sequence corresponding to $\{Z_g\}$ satisfies
	\[
	a V_n^\epsilon \leq b_n^\prime\leq b^Z_n = \left(b_n^\alpha + {b_n^\prime}^\alpha \right)^{1/\alpha}=o(V_n^{1/\alpha})
	\]
	because of vanishing of ECG and \eqref{condn_on_Xprime}. Using the inequality $b^Z_n \geq a |B_n|^\epsilon$, the series representation \eqref{series_repn_1} and the arguments given in the proof of (4.3) in \cite{samorodnitsky:2004a}, it follows that $M_n / b^Z_n$ is \emph{stochastically bounded} (also known as \emph{tight}), i.e., given any $\eta \in (0,1)$ there exists $K = K(\eta) >0$ such that
	\[
	\inf_{n \geq 1} \mathbb{P}\left(\bigg|\frac{M_n}{b^Z_n}\bigg| \leq K\right) > 1 - \eta.
	\]
	This, together with $b^Z_n = o(|B|_n^{1/\alpha})$, yields \eqref{convprob_to_zero}.
\end{proof}

\noindent \textbf{Heuristics and idea behind \eqref{main_step_for_heu}:}
Instead of rewriting in detail the proof of Equation (4.9) in \cite{samorodnitsky:2004a}, we provide the heuristics behind it. The main tool for verifying \eqref{main_step_for_heu} is, as expected, the series representation \eqref{series_repn_1} mentioned above.  The heuristics behind this are based on the \emph{one large jump principle}, which can be described as follows. It can be shown that
\[
\mathbb{P}\left( \left|b_n \mathfrak{C}_\alpha^{1/\alpha} \varepsilon_1 \Gamma_1^{-1/\alpha} \frac{f_g(U^{(n)}_1)}{\max_{h \in B_n}|f_h(U^{(n)}_1)|}\right|> \lambda\right) \sim c_0 \lambda^{-\alpha}
\]
for some $c_0 >0$ as $\lambda \to \infty$ whereas
\[
\mathbb{P}\left( \left|b_n \mathfrak{C}_\alpha^{1/\alpha} \sum_{j=2}^{\infty} \varepsilon_j \Gamma_j^{-1/\alpha} \frac{f_g(U^{(n)}_j)}{\max_{h \in B_n}|f_h(U^{(n)}_j)|}\right|> \lambda\right) =o(\lambda^{-\alpha}).
\]
See Pages 26-28 of \cite{samorodnitsky:taqqu:1994}. According to the discussion on Page 26 of this reference, the first term of \eqref{series_repn_1} is the dominating term that gives the precise asymptotics of its tail while the rest of the terms provide the ``necessary corrections'' for the whole sum to have an $\sas$ distribution.

In light of the above one large jump heuristics, we get that for all $\lambda >0$,
\begin{align*}
	\mathbb{P}\left(\frac{M_n}{b_n} > \lambda\right) &= \mathbb{P}\left(\max_{g \in B_n} \left|\mathfrak{C}_\alpha^{1/\alpha} \sum_{j=1}^\infty \varepsilon_j \Gamma_j^{-1/\alpha} \frac{f_g(U^{(n)}_j)}{\max_{h\in B_n} f_h(U^{(n)}_j)}\right|\, > \, \lambda\right),\\
	&\approx \mathbb{P}\left(\max_{g \in B_n} \left|\mathfrak{C}_\alpha^{1/\alpha}\varepsilon_1 \Gamma_1^{-1/\alpha} \frac{f_g(U^{(n)}_1)}{\max_{h \in B_n} f_h(U^{(n)}_1)}\right|\, > \, \lambda\right)\\
	&=\mathbb{P}(\mathfrak{C}_\alpha^{1/\alpha} \Gamma_1^{-1/\alpha} > \lambda) = 1 - e^{-\mathfrak{C}_\alpha \lambda^{-\alpha}}.
\end{align*}
This  computation yields \eqref{main_step_for_heu}. The key step (namely, the ``$\approx$'' above) can be made precise with the help of \eqref{series_repn_1} and the language of Poisson random measures; see Pages 1454 - 1455 of \cite{samorodnitsky:2004a} for details.

\begin{rmk}
	Theorem \ref{main_prob} above implies that $M_{n} / V_{n}^{1/\alpha}$ is \emph{stochastically bounded} (also known as \emph{tight}) and is ``bounded away from zero'' as long as ECG is non-vanishing.
\end{rmk}

\section{Patterson-Sullivan-Bowen-Margulis measures}\label{sec:psbm}
\subsection{Existence of Quasiconformal densities}\label{sec:ps} Let $(X,d)$ be a  proper geodesic metric space with base-point $o$ and equipped with
a  properly discontinuous isometric action of a group $G$.
\begin{defn}
	The {\bf 2-variable Poincar\'e series}   is the sum $$P_s(x,y) :=\sum_{g\in\G} e^{-s d(x,g(y))}.$$ For $x=y=o$, $P_s(o,o) = P(s)$ will simply be called the Poincar\'e series.
\end{defn}
In all the cases of interest in this paper, there exists $v >0$, called the {\bf critical exponent} such that for all $s>v$, the \poin series converges and for all $s<v$, the \poin series diverges.
\subsubsection{Patterson-Sullivan Measures for hyperbolic spaces}\label{sec:psm}
We refer the reader to  \cite{patterson-acta,sullivan-pihes,coornert-pjm} for the construction  of Patterson-Sullivan measures when $X$ is Gromov-hyperbolic.
The {\bf limit set} $\Lambda_G (\subset \partial X)$ of the  group $G$ acting on $X$ is the collection of accumulation points in $\partial G$ of a $G$-orbit $G.o$ for some (any) $o \in X$.
The group $G$ acts by homeomorphisms on $\Lambda_G$ given by $\phi_g(x)=g^{-1}\cdot x$. We shall represent this action as $g \longrightarrow \phi_g$. This is consistent with the action in Equation \ref{eq-anxi}: $\mu \circ \phi_g = \mu_{g^{-1}.o}$.

\begin{theorem}\label{coornaert}\cite{coornert-pjm}
	Let $(X,d)$ be a  proper Gromov-hyperbolic metric space equipped with
	a  properly discontinuous (not necessarily convex cocompact) isometric action of a group $G$. Then there exists a quasiconformal density of dimension $v$ (equal to the critical exponent) supported
	on the limit set $\L=\lag$. Also, $v = \limsup_n \frac{1}{n} \log  |B_n|$, where $B_n$ is as in Section \ref{subsec-ecg}.
\end{theorem}

The quasiconformal density constructed by Coornaert in Theorem \ref{coornaert} is called the {\bf Patterson-Sullivan density}. When $P_s(x,y)$ diverges, the
{\bf Patterson-Sullivan measure} based at $o$ is obtained as a weak limit of the measures $$\sum_{g \in B_n} e^{-sd(x,g.y)}\dirac_{g.y}$$ normalized by $P_s(x,y)$ (see \cite{calegari-notes,cm-gafa} for details).
When $P_G(s)$ converges, an extra weighting function is introduced in front of the
exponential factors to force the modified $P_s(x,y)$ to diverge \cite{patterson-acta}. Note that $V_n$ (Equation \ref{eq-vn}) can be identified with the {\it volume growth} of balls of radius $n$ in the weak  hull $CH(\L)$ of $\L$ in $X$, where  $CH(\L)$ consists of the union of geodesics with end-point in $\L$.

\subsubsection{Patterson-Sullivan Measures for symmetric spaces of higher rank}\label{sec:psmss} In this subsection, $X$ will be a  symmetric space of noncompact type and $G$ a lattice. The visual or geometric boundary will be denoted as $\partial X$, while the Furstenberg boundary will be denoted as $\partial_F X$. The critical exponent of the Poincar\'e series is denoted by $v$ as before. The Furstenberg boundary  $\partial_F X$ can be naturally identified with the orbit of the centroid of a Weyl chamber in $\partial X$. Thus, $\partial_F X \subset \partial X$. As before, the action of $g$ on $\partial_F X$ will be denoted by $g \to \phi_g$.
Albuquerque shows (see Definition \ref{qcdens}):

\begin{theorem} \cite{albgafa}\label{alb} For $(X,G)$ as above, there exists a unique {\bf conformal} density given by the Patterson-Sullivan measure class $\{\mu_\xi\}$ supported on $\partial_F X \subset \partial X$.
\end{theorem}

\subsubsection{Thurston Measure for Teichm\"uller space}\label{sec:psmteich}
In this subsection, $X$ will denote the Teichm\"uller space $Teich(S)$ of a surface and $G=MCG(S)$ its mapping class group. The Thurston boundary, or equivalently, the space $\PMF (S)$ of projectivized measured foliations, will be denoted as $\partial X$. Let $\xi \in \partial X$ be a measured foliation. Let ${ Ext}_{\xi}(x)$ denote the extremal length at $x \in X$ of a measured foliation $\xi$.
The Thurston measure on the space of measured foliations $\MF(S)$ is denoted as $\mu$. For $\xi \in \MF(S)$, $[\xi]$ will denote its  image  in $\PMF(S)$. For $U \subset \PMF (S)$, the authors of \cite{abem} define a measure $\mu_x$ with base-point $x \in X$ as follows:
$$\mu_x(U)=\mu(\{\xi\, \vert \, [\xi]\in U\, , Ext_{\xi}(x)\leq 1\}).$$
Further, \cite[p. 1064]{abem}
$$
\frac{d\mu_x}{d\mu_y}([\xi])=\left(\frac{\sqrt{{ Ext}_{\xi}(y)}}{\sqrt{{ Ext}_{\xi}(x)}}\right)^{6g-6}.
$$

For $[\xi] \in \PMF (S)$, a Busemann-like
cocycle $\beta_{\xi}: Teich(S) \times Teich(S) \rightarrow \reals$  for the Teichm\"uller metric is defined as follows:
\[\beta_{\xi}(x,y)= \log\left(\frac{\sqrt{Ext_{\xi}(x)}}{\sqrt{Ext_{\xi}(y)}}\right).\]
This makes the family
$\{\mu_x\}_{x \in X} $ of probability measures on $\PMF(S)$ into a
family of $G-$invariant conformal densities of dimension $v=dim(X)$ for the cocycle
$\beta$:

\begin{theorem} \cite{abem}\label{abem} Let $(X,G), \mu_x, v$ be as above.
	For all $x, y \in X$ and $\mu_x-$almost every $\xi \in \partial X$,
	\begin{equation}
	\frac{d \mu_x}{d\mu_y} ([\xi]) = exp(v \,  \beta_\xi \, (y, x)).
	\end{equation}
	Further, for $g \in G$, $U \subset \PMF(S)$
	$$ \mu_{g.x }(g \cdot U)= \mu_x(U).$$
\end{theorem}

We shall refer to any $\mu_x$ above as a {\bf Thurston conformal density}.

\begin{rmk}\label{th=ps}
	As observed at the end of Section 2.3 of \cite{abem}, it follows from
	\cite[Theorem 2.9]{abem} that the family of measures $\mu_x$ give a Patterson-Sullivan density on $\partial X$.
\end{rmk}

\subsection{Bowen-Margulis measures}\label{baderfurman} We shall recall the construction of Bowen-Margulis measures by Furman \cite{furman-coarse} and Bader-Furman \cite{bader-furman} for $X$ Gromov-hyperbolic and $G$ acting properly discontinuously by isometries on it. This is slightly more general than what we need in most of the applications.
We will apply it in particular to CAT(-1) spaces (see \cite{roblin-memo} for an excellent treatment in the latter context).
Let $[\mu]$ be  the Patterson-Sullivan measure class of Theorem \ref{coornaert}. The square class $[\mu\times\mu]$ is supported on
$
\lagg := \{ (x,y) \in (\lag \times \lag)\vert x \neq y \}.
$
The authors of  \cite{bader-furman} state the Proposition below in the context of cocompact group actions but the proof goes through in the general case.

Let $\langle{x},{y}\rangle_{o}$ denote the Gromov inner product.
\begin{prop}\label{coarseBMS}\cite[Proposition 1]{furman-coarse}\cite[Proposition 3.3]{bader-furman}
	There exists a $G-$invariant Radon measure, denoted $\mubms$, in the measure class $[\mu\times \mu]$
	on $\lagg$. Moreover, $\mubms$ has the form
	\[
	d\mubms(x,y)=e^{F(x,y)}\,d\mu(x)\,d\mu(y)
	\]
	where $F$ is a measurable function on $(\lagg,[\mu\times\mu])$ of the form 	
	$
	F(x,y)=2v\, \langle{x},{y}\rangle_{o}+O(1).
	$
\end{prop}

Let $\LL$ denote the Lebesgue
measure on $\R$.
Bader and Furman extend the ergodic $G-$action on $(\lagg,\mubms)$ to a $G-$action on $(\lagg\times \R, \mubms\times \LL)$ as follows. Let $\Phi^\R$ denote the $\R-$action on $\lagg\times \R$ given by $\Phi^s(x,y,t)=(x,y,t+s)$.

\begin{prop}\cite[Proposition 3.5]{bader-furman}\label{almost2exact}
	The $G-$action on $(\lagg,\mubms)$ of Proposition \ref{coarseBMS}  extends to
	a $G-$action on $(\lagg\times \R, \mubms\times \LL)$ given by $(x,y,t) \rightarrow g \cdot (x,y,t)$ satisfying the following:
	\begin{enumerate}
		\item $G$ preserves the infinite measure $\mubms\times \LL$.
		\item The $G-$action commutes with the $\Phi^\R$-action.
		\item The $G-$action commutes with the flip: $(x,y,t)\mapsto (y,x,-t)$.
	\end{enumerate}
\end{prop}
The measure-preserving action of $G \times \R$ on $(\lagg\times \R, \mubms\times \LL)$  induces a flow $\phi^\R$ on the quotient measure space $$(UM,\mubm):= (\laggr,\mubms \times \LL)/G.$$
We call $(UM,\mubm)$  the {\bf measurable unit tangent bundle} corresponding to the action of $G$ on $X$; and $\mubm$  the {\bf Bowen-Margulis measure} on the  measurable unit tangent bundle $UM$.

The two variable {\bf growth function} is given as follows: $$V_G(x,y,n) = \#\{g \in G: d(x,gy) \leq n\}.$$
We refer the reader to \cite{roblin-memo} for an excellent introduction to Patterson-Sullivan and Bowen-Margulis measures in the context of CAT(-1) spaces. We shall say that a group action of $G$ on $X$ has {\bf non-arithmetic length spectrum} if there does not exist $c>0$ such that all translation lengths are integral multiples of $c$. Roblin \cite{roblin-memo} proved the following dichotomy for group actions on $CAT(-1)$ spaces.
\begin{theorem}\cite[Chapter 4]{roblin-memo}\label{roblindich} Let $G$ be a discrete
	non-elementary group of isometries
	of a CAT(-1) space $X$ with non-arithmetic length spectrum and critical exponent $v$. Then  one has  one of the following two alternatives:
	\begin{enumerate}
		\item There exists a function $c_G: X \times X \to \R_+$ such that
		$V_G(x,y,n) \asymp
		c_G(x, y)e^{vn}$ if the Bowen-Margulis measure of the measurable unit tangent bundle  is finite: $\mubm(UM) < \infty$.
		\item  $V_G(x,y,n) = o(e^{vn})$ else.
	\end{enumerate}
\end{theorem}

The proof of Theorem \ref{roblindich} above depends crucially on mixing of the geodesic flow in this context \cite[Chapter 3]{roblin-memo}.
In this strong form, it fails for hyperbolic groups equipped with the word metric (see the discussion after \cite[Corollary 1.7]{bader-furman}, where the authors prove a weaker version of mixing).

\section{Mixing and equidistribution of spheres}\label{sec:mix}
Let $(\Omega,m)$ be a finite measure space. Let $\GG$ be a
a locally compact topological group  acting on $(X,m)$ preserving $m$. The $\GG-$action is said to be  {\bf mixing}  if, for any pair of measurable subsets $A, B \subset \Omega$, and any sequence $g_n \to \infty$ in $G$,  $$\nu(A \cap g_n B) \to \frac{\nu(A) \, \nu(B)}{\nu(\Omega)}.$$

\subsection{CAT(-1) spaces}\label{sec-mixcat} For this subsection $X$ is a proper CAT(-1) space, and $M=X/G$. We can construct the {\it geometric} tangent bundle to $M$ as follows: $$U_gM = (X \times \partial X)/G,$$ where $G$ acts diagonally. In this context, Roblin \cite{roblin-memo} constructs the Bowen-Margulis measure $\mubm$ on $U_gM$ converting it to a space measure-isomorphic to 
the measurable unit tangent bundle $(UM,\mubm)$ (described just after Proposition \ref{almost2exact}).

When  $\mubm(UM) < \infty$, Roblin \cite[Chapter 3]{roblin-memo} proves that the Bowen-Margulis measure $\mubm$ is  mixing under the geodesic flow on $UM$ unless the length spectrum is arithmetic (see also \cite{ricks,glink18}). Conjecturally, arithmetic length spectrum
is equivalent to the condition that there exists $c > 0$ such that $X$ is isometric to a tree with all edge lengths in $c\natls$ (this has been proven under the additional assumption that the limit set of $G$ is full, i.e.$\lag = \partial X$ in \cite{ricks}).
Let $P:UM\to M$ be the natural projection, so that $P^{-1}(p) =S_p$ may be thought of as the {\it `unit tangent sphere'} at $p \in M$. $S_p$
can be naturally identified with the boundary $\partial X$  equipped with the Patterson-Sullivan measure $\mu$ supported on the limit set $\lag$
(see for instance the discussion on skinning measures in \cite[Section 3]{pp-skin}). We denote this measure
by $\mu_p$ and think of it as the Patterson-Sullivan measure on $\lag$ based at $p$.  
Broise-Alamichel,
Parkkonen and Paulin \cite{pp-eq,pp-skin} (see also  \cite{eskin-mcmullen}) prove that when $X$ is CAT(-1) and the geodesic flow is mixing, then  $\mu_p$
equidistributes to the Bowen-Margulis measure (see also \cite{os-inv,os-jams}
where the considerably more general notion of skinning measures was introduced).
Let $A \subset UM$ denote any measurable subset and let
$A_{p,t} :=    \{x \in S_p \vert g_t(x) \in A\}.$
We summarize these results below:

\begin{theorem}\cite{eskin-mcmullen,pp-skin,pp-eq,ricks}\label{em}
	Suppose $X$  is $CAT(-1)$ such that $M = X/G$ has non-arithmetic length spectrum. For  $UM$ as above,
	suppose $\mubm(UM) < \infty$. Then
	the sphere $S(p,t)$ of radius $t$ about a point $p \in X/G$
	becomes equidistributed
	as $t \to \infty$ in the following sense. For any measurable subset  $A \subset UM$, and $p \in M$, $$\frac{\mu_p(A_{p,t})}{\mu_p(S_{p})}\rightarrow \frac{\mubm(A)}{\mubm(UM)}$$ as $t\to \infty$.
\end{theorem}

\subsection{Symmetric spaces}\label{sec-mixss}

Using results of Kleinbock-Margulis~\cite{KM1, KM2},
we will prove the following general statement:

\begin{theorem}\label{hom} Let $\GG$ be a connected semisimple Lie group without
	compact factors, let $K$ be a maximal compact subgroup with Haar measure
	$\nu$, and let $G$ be an irreducible lattice in $\GG$.  Let $A^+ \subset A$ denote the positive Weyl chamber $A^+$ in the Cartan subgroup $A$. Let $\{g_t\} \subset A^{+}$ denote a one-parameter subgroup and let $\mu$ denote the Haar measure on $M = K \backslash \GG/G$ inherited from
	$\GG$.  Let $d$ denote the distance function on $M$ arising from a right $\GG$-invariant Riemannian metric on $K \backslash \GG$ (for example, the metric induced by the Killing form on $\GG$). Let $\pi: \GG \rightarrow M$ denote the map $\pi(g) = KgG$. Then, for all $g, g_0 \in \GG$, we have $$\lim_{t \rightarrow \infty}
	\int_K d(\pi(g_tkg), \pi(g_0))d\nu(k) = \int_M d(x, y)d\mu(y).$$
\end{theorem}


\medskip

\noindent Our main tool is the following result of
Kleinbock-Margulis

\begin{theorem}~\cite[Corollary A.8]{KM2}: Fix notation as in Theorem~\ref{hom}. Let
	$\Omega = \GG/G$, and $\phi \in L^2(\Omega, \eta)$, where $\eta$ is the Haar measure on $\Omega$ 
	Assume that $\phi$ is H\"older continuous. Then $$\lim_{t \rightarrow \infty}
	\int_K \phi(g_tkx)d\nu(k) = \int_{\Omega} \phi(y)d\eta(y).$$\end{theorem}

\medskip
\noindent Theorem \ref{hom} is an immediate corollary of this result, since the function $\phi(g) = d(\pi(g), \pi(g_0))$ is clearly H\"older continuous on $\GG$ and therefore on $\Omega$. To check that it is
in $L^2$, we use~\cite[\S5]{KM1}, which shows that the tails of this distance function in fact decay \emph{exponentially}: there are $C_1, C_2 >0$ such that $$\eta\{ x \in \Omega: d(x, \pi(g_0) > t\} \le C_1 e^{-C_2t}.$$
Thus, we have Theorem~\ref{hom}. In fact~\cite{KM2} gives a precise estimate
on the rate of convergence, linking it to the exponential rate of mixing for the
flow $g_t$. In fact, this rate of convergence can be bounded below for  any $g_t = \exp(tz)$ where $z$ is in the norm $1$ subset of the positive Weyl chamber $\mathcal A^+$, and so by doing an extra integration over this set, we can get equidistribution of the \emph{whole} sphere in the space $M$. See~\cite[\S6]{KM1} for more details on describing the geodesic flow on symmetric spaces using the orbits of one-parameter subgroups, following ideas of Mautner.

\subsection{Teichm\"uller and moduli space}\label{sec-mixteich}

For the purposes of this subsection,
let $X=Teich(S)$ be the Teichm\"uller space of a surface $S$, $G=MCG(S)$ be its mapping class group, and $M = X/G$ the moduli space.
Let $UX$ (resp. $UM$) denote the  bundle of unit-norm holomorphic quadratic
differentials on $X$ (resp. $M$). Let $\pi: UM \rightarrow M$ denote the natural projection.
Let $g_t$ denote the Teichm\"uller geodesic flow on $UM$. Masur-Smillie~\cite{MasurSmillie} building on earlier work of Masur~\cite{Masur} and Veech~\cite{Veech}
showed that $UM$ carries a
unique  measure 
$\mu$ (up to scale) in the Lebesgue
measure class such that $\mu(UM) < \infty$ and $g_t$ is mixing. Let $\eta = \pi_* \mu$ denote push-forward of $\mu$ to $M$. For any $x\in M$, denote the  unit-norm holomorphic quadratic
differentials at $x$ by $S(x)$. Identify $S(x)$ with the Thurston boundary $\PML(S)$ of $Teich(S)$ and equip it with the Thurston measure $\nu_x$ based at $x$. Let $\eta_{t, x}= \pi_*g_{t}^* \nu_x$ denote the measure $\nu_x$ pushed forward to the sphere of radius $t$ $\{\pi(g_t(x,v))\} \subset M$, i.e., $$ d\eta_{t,x}(\pi(g_t(x,v)) = d\nu_x(v).$$

\begin{theorem}\label{teichmix} Fix $x_0 \in M$, For almost all $x \in M$, we have $$\lim_{t \rightarrow \infty}
	\int_{M} d(y, x_0) d\eta_{t,x}(y) = \lim_{t \rightarrow \infty} \int_{S(x)} d(\pi(g_t(x,v)), x_0)d\nu_x (v)= \int_M d(y, x_0)d\eta(y).$$  
\end{theorem}

\medskip

\noindent To prove this theorem, fix $(x, v) \in UM$. The group $SL(2, \R)$ acts on $UX$, and its action commutes with the mapping class group, so it acts on $UM$. The action of the group $$a_{t} = \left( \begin{array}{cc} e^{t}  & 0 \\ 0 & e^{-t}
\end{array}\right)$$ is precisely the geodesic flow $g_t$.  The circles $\{a_t r_{\theta}(x, v): 0 \le
\theta \le 2\pi \}$, where $$r_\theta = \left( \begin{array}{cc} \cos
\theta & \sin \theta \\ -\sin \theta & \cos \theta \end{array}\right)$$
foliate the sphere of radius $t$ around $x \in X$. Let $K = \{r_{\theta}: 0 \le \theta < 2\pi\}$ denote the maximal compact subgroup of $SL(2, \R)$. Let $d\kappa(\theta) = \frac{1}{2\pi} d\theta$ on $K$, and let $d\kappa_{x,v}(r_{\theta}(x,v)) = d\kappa(\theta)$,and let $\kappa_{x, v, t} = a_t^* \kappa(x,v).$

To prove Theorem~\ref{teichmix}, we will use the following ergodic theorem of Nevo's~\cite[Theorem 1.1]{Nevo}, which in our case implies:

\begin{theorem}\label{theorem:nevo}\cite[Theorem 1.1]{Nevo} Let $f \in L^{2}(UX)$ be $K$-finite, that is, the span of the set of functions $\{f_\theta(x, v) = f(r_{\theta}(x,v)): 0 \le \theta < 2\pi\}$ is finite dimensional. Then for $\mu$-almost every $(x, v)$, \begin{equation}\label{eq:Nevo} \int_{UM} f d a_t\nu(x, v) \xrightarrow{t \rightarrow \infty} \int_{UM} f d\mu.\end{equation}
	
\end{theorem}

\medskip

\noindent To use this theorem for our result, we note that there is a measure $\omega$ on $\tilde{S}(x) = S(x)/K$ so that we can write $$d\eta_{t,x} = \int_{\tilde{S}(x)} da_t\nu(x, [v]) d\omega([v]).$$ We also note that the function $f(x, v) = d(\pi(x, v), x_0)$ is $K$-invariant, and by the following lemma of Masur~\cite{masurloglaw}, in $L^2(UX, \mu)$. Let $\ell(x, v)$ denote the length of the shortest saddle connection of $(x, v)$ (recall that a \emph{saddle connection} is a geodesic in the flat metric on $S$ determined by the quadratic differential $(x, v)$ joining two zeroes, with no zeroes in its interior).

\begin{lemma} There is a constant $C'$ such that for any $x_0 \in M$
	and $(x, v) \in UM$, we have $$ d(\pi(x, v), x_0) = d(x, x_0)
	\le -\log \ell(x,v) + C'.$$
\end{lemma}

\medskip

\noindent By Masur-Smillie~\cite{MasurSmillie}, $$\mu\{ (x, v): \ell(x, v) < \epsilon\} \sim \epsilon^2,$$ which, combined with the lemma, yields that $f \in L^2(UX, \mu)$. To finish the proof Theorem~\ref{teichmix}, we note that if there was a positive $\eta$-measure set of $x \in M$ so that the set of $[v] \in \tilde{S}(x)$ with $$\int_{UM} f d a_t\nu(x, v) \nrightarrow \int_{UM} f d\mu$$ had positive $\omega$-measure, we would have a set of positive $\mu$-measure in $UX$ where (\ref{eq:Nevo}) fails, a contradiction.

\section{Extremal cocycle growth and spherical averages}\label{sec:ecg-sa} In this section, we shall establish a connection between  extremal cocycle growth as in Definition \ref{def-ecgps} and the asymptotics of spherical averages. This will, in particular, allow us to apply the equidistribution theorems  of the previous section. 
\subsection{Averaging measures and spherical averages}
Let $G,X \L, \mu$ be as in Setup \ref{setup} and $o \in X$ be a base-point. Let $\sro=\partial B(o,r)$ denote the boundary of the $r-$ball about $o$. If $X$ is CAT(0) or $Teich(S)$, there is a natural family of continuous projection maps $\pi_{r,t}: \sro \to \s_t(o)$ for $r>t$ sending $x \in \sro$ to $[o,x] \cap \s_t(o)$. We also have the natural projections $\pi_r: \sro \rightarrow \partial X$.

\begin{defn}\label{def-avgingmreconf} Let $X$ be CAT(0) or $Teich(S)$.
	A sequence of probability measures $\{\mu_r\}$ on $\sro$ is said to be a 
	sequence of {\bf averaging measures with respect to a conformal density} $\mu$ supported on $\L \subset \partial X$ if
	\begin{enumerate}
		\item $\pi_{r, t \ast} (\mu_r) = \mu_t$, for $r>t$,
		\item there exists $C\geq 1$ such that if $\mu_\infty$ (a measure on $\partial X$) is any  weak limit of $\pi_r*(\mu_r)$ up to subsequences, then $\mu_\infty$ is supported on 
		$\L$ and $1/C \leq \frac{d\mu_\infty}{d\mu} (\xi) \leq C, \, \forall \xi \in \L$.
	\end{enumerate}
	If $C=1$, then $\{\mu_r\}$ is said to be {\bf strongly averaging}.
\end{defn}

When $X$ is a uniformly proper $\delta-$hyperbolic graph with all edges of length one, $\pi_{rt}$ is not well-defined, but only coarsely so. Thus, for $x \in \sro$, we define  $\pi_{rt} (\dirac_x) $ to be the uniform probability distribution on the set $$\{y \in \s_t(o) \vert \exists \, {\rm geodesic} \, \gamma \, {\rm such \, that} \, o, x, y \in \gamma\}.$$
Note that for $r>t$, the support of $\pi_{rt} (\dirac_x) $ has diameter at most $\delta$. 

\begin{defn}\label{def-avgingmreqc} Let $X$ be a uniformly proper $\delta-$hyperbolic graph with all edges of length one.
	A sequence of probability measures $\{\mu_r\}$ on $\sro$ is said to be a 
	sequence of {\bf averaging measures with respect to a quasiconformal density}
	$\mu$ supported on $\L \subset \partial X$ if
	there exists $C\geq 1$ such that
	\begin{enumerate}
		\item $1/C \leq \frac{d(\pi_{r,t\ast}\mu_r)}{d\mu_t} (x) \leq C, \, \forall x \in \s_t(o)$, whenever $r >t$.
		\item  if $\mu_\infty$ is any  weak limit of $\mu_r$ up to subsequences, then $\mu_\infty$ is supported on 
		$\L$ and $1/C \leq \frac{d\mu_\infty}{d\mu} (\xi) \leq C, \, \forall \xi \in \L$.
	\end{enumerate}
\end{defn}

Note that in Definitions \ref{def-avgingmreconf} and \ref{def-avgingmreqc}, the projections $\pi_{sr}$ for fixed $r$ and $s>r$ can be extended to a projection $\pi_r : \L \to \s_r(o)$ such that

\begin{equation}\label{eq-projmre}
1/C  \leq  \frac{d(\pi_{r\ast}\mu)}{d\mu_r} (x) \leq  C, \, \forall x \in \s_r(o).
\end{equation}

\begin{eg}[Examples of Averaging Measures]\label{egs}
	{\rm We enumerate the examples of interest:\\
		
		\noindent 1) For $(G,X,\L)$ as in Item (1) of Setup \ref{setup}, let $\mu$ be a Patterson-Sullivan density as in Theorem \ref{coornaert}. Assume further that the  associated  Bowen-Margulis measure is finite. Let $\mu_r$ be the 
		the conditional of the Bowen-Margulis measure (equivalently, the Patterson-Sullivan measure
		on $\lag$	based at $o$) pushed forward by the geodesic flow for time $r$.\\

		\noindent 2) For $(G,X,\L)$ as in Item (2) of Setup \ref{setup}, let $\mu$ be a Patterson-Sullivan density as in Theorem \ref{coornaert}. Assume further that the  associated  Bowen-Margulis measure is finite.  Join $o$ to all points $p \in \L$ by geodesic rays to obtain  the cone over the limit set denoted as $\qcg_o$. Note that $\qcg_o$ is $2 \delta-$quasiconvex. Let $\mu_r$ be the uniform distribution on 
		$\qcg_o \cap \sro$ (the equivalence of the uniform measure and the measures on "cylinder sets" is explicitly stated in \cite[Proposition 3.11]{calegarimaher}).\\

		\noindent 3) For $(G,X,\L)$ as in Item (3) of Setup \ref{setup}, let $\mu$ be the Thurston density as in Theorem \ref{abem}. Identifying $\L=\PMF(S) $ with the unit norm quadratic differentials $\QQ^1_o$ at $o$, define $\mu_r$ on $\sro$ to be $\mu$ pushed forward by the Teichm\"uller geodesic geodesic flow for time $r$. Note that by Remark \ref{th=ps}, $\mu$ is a Patterson-Sullivan density.\\

		\noindent 4) For $(G,X,\L)$ as in Item (4) of Setup \ref{setup}, let $\mu$ be the Patterson-Sullivan density as in Theorem \ref{alb}. Let $z$ denote the barycenter of a Weyl chamber at infinity and $[o,z)$ the geodesic ray from $o$ to $z$. Let $z_r \in [o,z) $ be such that $d(o,z_r) = r$. For $K$ the maximal compact of the semi-simple Lie group $\GG$ (with $K\backslash\GG = X$), let $\mu_r$ be the uniform measure on $K_r := z_r\cdot K$
		inherited from the Haar measure on $K$. Note that $K_r \subset \sro$.}
\end{eg}

\begin{prop}\label{prop-avm} Let $\mu_r, \mu$ be one of the four examples in \ref{egs}. Then $\{\mu_r\}$ is a sequence of averaging measures with respect to $\mu$. 
\end{prop}

\begin{proof}
	Example (1): This follows from the construction of the Bowen-Margulis measure in Propositions \ref{coarseBMS} and \ref{almost2exact}.\\
	Example (2): This follows from the construction of the Patterson-Sullivan measure \cite{coornert-pjm} when the \poin series diverges. Else $|B_n| \asymp exp(vn)/g(n)$ for a subexponentially growing function $g$. In this case, Patterson's trick \cite{patterson-acta,sullivan-acta84}  of multiplying the terms of the \poin series by $g(n)$ gives back the growth function of $\qcg_o \cap \sro$. A Patterson-Sullivan measure is then obtained as a weak limit of the uniform distribution on $\qcg_o \cap \sro$.\\
	Example (3): This follows from the construction of the  measure $\mu_x$ from the Thurston measure $\mu$ in Section \ref{sec:psmteich} \cite{abem}.\\
	Example (4): This follows from \cite[Theorem C, Proposition D, p. 4]{albgafa}.
\end{proof}

\begin{defn}\label{def-sa} {\rm Let $(G,X,\L)$ be one of the four examples \ref{egs}. Let $M=X/G$
		and $y_0\in M$ be a base-point.
		Let $f_0: M \to \R_+$ be a function. Let $f$ be the lift of $f_0$ to $X$ and $o$ be a lift of $y_0$. 
		The family of expectations $\{\E_r(f) = \int_{\sro} f d\mu_r \}$   will be called the {\bf spherical averages} for the triple $(X,G,f)$ with respect to the base-point $o$.
		
		Let $v$ denote the  dimension of the conformal or quasiconformal density  $\mu$ on $\L$. For $f_0: M \to \R_+$ given by $f_0(w) =  exp(-v d_M(y_0,w))$,  $\fex=exp(-v \, d_X(x,G.o))$ will denote the lift of $f_0$ to $X$. The  spherical averages $\{\E_r(\fex)\}$ for  $(X,G,\fex)$, given by 
		\begin{equation}\label{eq:fex}
		\E_r(\fex) = \int_{\s_r(o)} exp(-v \, d_X(x,G.o)) d\mu_r(x) = \int_{\s_r(o)} \fex (x) \, d\mu_r(x),
		\end{equation}
		will be called {\bf extremal spherical averages} with respect to the base-point $o$.}
\end{defn}

Note that the domain of $ \fex$ is $X$ (and hence it can be integrated over $\s_r(o)$).

\subsection{A sufficient condition for non-vanishing ECG}
Recall that $V_r =  e^{vr}$ (Equation \ref{eq-vn}) and $B_r = \{g \in G \vert g.o \in B(o,r) \}$. We write $B_r.o= \{g.o \ \vert g \in B_r \}$ and define:
\begin{equation}\label{eq-fexr}
\fexr (\xi) = exp(-v \, d_X(\pi_r(\xi), B_r.o)), \, \xi \in \L,\end{equation}
\begin{equation*}
\fexr (x) = exp(-v \, d_X(\pi_{sr}(x), B_r.o)), \, x \in X, \, d_X(o,x) = s\geq r.
\end{equation*}
The domain of $ \fexr$ is $X \cup \partial X$ and hence it can be integrated over both $\s_s(o)$ and $\lag$. Define
\begin{equation}\label{eq:efexr}
\E_r(\fexr) = \int_{\s_r(o)} exp(-v \, d_X(x,B_r.o)) d\mu_r(x) = \int_{\s_r(o)} \fexr (x) \, d\mu_r(x),
\end{equation}

We shall write $\omrk = 
(B_{r+k}.o \setminus B_r.o)$ to denote the collection of orbit-points in the shell $(B(o,r+k) \setminus B(o,r))$.\\

\noindent{\bf Relating $\fex$ and $\fexr$:}
\begin{lemma}\label{lem-fexfexr}
	Let $\{\mu_r\}$ be averaging measures on $\sro$  with respect to $\mu$ as in the four examples \ref{egs}. For any $m > 0$ and $k \in \natls$, there exist $K, R_0 \geq 1$ such that for $r \geq R_0$,
	\begin{align*}
	K \mu_r (\{x\in\sro \vert d(x, B_r.o)\leq m\}) &\geq& \mu_r (\{x\in\sro \vert d(x, B_{r+k}.o)\leq m\}) \\
	&\geq&\mu_r (\{x\in\sro \vert d(x, B_r.o)\leq m\}).
	\end{align*}
\end{lemma}

\begin{proof}
	The second inequality is clear and we need only to prove the first.
	We use the fact that the measure $\mu$ in all 
	the examples in \ref{egs} are Patterson-Sullivan densities.  In particular, it follows from the construction of Patterson-Sullivan measures  that for every $k \in \natls$, there exists $C_0 \geq 1$ such that if $\nu^k$ is any limit of uniform probability 
	distributions on the $k-$shells
	$\omrk$ as $r\to \infty$, then $1/C_0 \leq \frac{d\nu^k}{d\mu} (\xi) \leq C$ for all $\xi \in \L$. Hence the uniform probability distributions on the `inner' $k-$shell $\omrkm$ and the `outer' $k-$shell $\omrk$ about $\sro$ are close to each other: more precisely any two  limits of uniform
	distributions on $\omrkm$ and  $\omrk$ are absolutely continuous with respect to each other with pointwise Radon-Nikodym derivative lying in $[1/C^2, C^2]$.
	
	Now, we use the fact that $\{\mu_r\}$ is a sequence of averaging measures.  We argue by contradiction. Suppose that for some fixed $m$,  no $K \geq 1$ exists as in the conclusion of the Lemma. We pass to the limit as
	$r\to \infty$. Extracting subsequential limits if necessary,  there exists a limit $\mu_-$ of the inner shell measures, a  limit $\mu_+$ of the outer shell measures and a measurable subset $U \subset \L$ such that
	$\mu_+ (U) =0$ while $\mu_- (U) > 0$.
	This contradicts the absolute continuity in the last sentence of the previous paragraph, proving the Lemma.
\end{proof}

\begin{cor}\label{cor-fexfexr}
	Let $\{\mu_r\}$ be averaging measures on $\sro$  with respect to $\mu$ as in \ref{egs}. For any $\ep > 0$  there exists $K, R_0 \geq 1$ such that for $r \geq R_0$,
	\begin{align*}
	K \mu_r (\{x\in\sro \vert \fexr (x) \geq \ep\}) &\geq& \mu_r (\{x\in\sro \vert \fex (x) \geq \ep\}) \\
	&\geq&\mu_r (\{x\in\sro \vert \fexr (x) \geq \ep\}).
	\end{align*}
\end{cor}
\begin{proof}
	Choose $k\in \natls$ such that $e^{-kv} \leq  \ep < e^{-(k-1)v}$ and let $m=k$. For  $x \in \sro$,  $\fex(x) \geq \ep$  implies $d(x, B_{r+k}.o)\leq m$. The Corollary now follows from Lemma \ref{lem-fexfexr}.
\end{proof}

We are now in a position to state a sufficient condition guaranteeing non-vanishing ECG.

\begin{prop}\label{prop-crucialsuffcond} Let $(G,X,\L)$ be as in cases
	1,3, or 4 of setup \ref{setup} with $M=X/G$ and $o\in X$ a base-point. Let $P: X \to M$ be the quotient map, $P(o)=y_0$. Let $\{\mu_r\}$ be averaging measures on $\sro$  with respect to $\mu$ as in \ref{egs}.
	Suppose that there exists $c, R_0 >0$ and  $1\geq \alpha >0$ such that for all $r \geq R_0$,  $$\mu_r(\{x\in \sro \vert d_M(P(x),y_0) \leq c \}) \geq \alpha.$$ Then the action of $G$ on $(\L,\mu)$  has non-vanishing
	extremal cocycle growth.
\end{prop}

\begin{proof} Let $\ep = e^{-cv}$.
	It follows by hypothesis, that for all $r \geq R_0$, $$\mu_r (\{x\in\sro \vert \fex (x) \geq \ep\}) \geq \alpha.$$  Hence by Corollary \ref{cor-fexfexr}, there exists $K \geq 1$ such that for all $r \geq R_0$
	$$\mu_r (\{x\in\sro \vert \fexr (x) \geq \ep\}) \geq \alpha/K.$$ Let
	$\sro(\ep):=\{x\in\sro \vert \fexr (x) \geq \ep\}$ and let $\L(r,\ep):=
	\{\xi \in\L \vert \pi_r (\xi) \in \sro(\ep)\}$. Since $\{\mu_r\}$ is a family of averaging measures, there exists $K_1, R_1 \geq 1$ such that for all $r \geq R_1$, $\mu(\L(r,\ep) \geq \alpha/K_1.$ Hence, from Equation \ref{eq:ecg}, there exists $K_2 \geq 1$ such that for all $r \geq R_1$, $$C_r \gesim  (\alpha/K_2) \ep.$$
	Thus, $\liminf_{r \to \infty} C_r >0$; equivalently, the action of $G$ on $(\L,\mu)$  has non-vanishing
	extremal cocycle growth.
\end{proof}

\subsection{ECG for hyperbolic spaces}
When $X$ is Gromov-hyperbolic, i.e.\ Cases 1, 2 of Setup \ref{setup} we can say more.
\begin{prop}\label{cr} Let $(G,X,\L)$ and $\mu$ be as in Cases 1, 2 of Setup \ref{setup}. Let $\{\mu_r\}$ be a family of averaging measures as in \ref{egs}. Then there exists $R_0$ such that for $r \geq R_0$,
	\begin{equation}\label{eq:fexr}
	C_r  \asymp \int_\L \fexr(q) d\mu(q) \asymp  \int_{\s_r(o)} \fexr(x) d\mu_r(x).
	\end{equation}
\end{prop}

To prove Proposition \ref{cr} we shall need the following basic Lemma from hyperbolic geometry (see, for instance, \cite[Lemma 3.3]{mitra-trees} or \cite[Lemma 3.3]{mitra-ct} for a proof):

\begin{lemma}\label{npp} Given $\delta, C \geq 0$ there exists $C'$ such that the following holds:\\
	Let $X$ be a $\delta-$hyperbolic metric space and $K\subset X$ be $C-$quasiconvex. For any $p \in X\setminus K$, let $\pi_K(p)$ denote a nearest point projection of $p$ onto $K$ and let $[p,\pi_K(p)]$ be the geodesic segment joining $p, \pi_K(p)$. For  $k \in K$, let $[\pi_K(p),k]$ be the geodesic segment joining $\pi_K(p),k$. Then $[p,\pi_K(p)]\cup [\pi_K(p),k]$ is a $(C',C')-$quasigeodesic.
\end{lemma}

For $q \in \L$ and  $q' \in [o,q)$, $\pi_{r}(q')$ being uniformly close to $\pi_r(q)$,  we have the following  consequence of Lemma \ref{npp} using the fact that the balls $B(o,r)$ are $\delta-$quasiconvex. For
$q \in \partial X$ and $a \in X$, the geodesic ray from $a$ to $q$ is denoted as $[a,q)$.

\begin{cor}\label{nppfrominfty}
	Given $\delta>0$ there exists $C>0$ such that the following holds for any $r>0$: If $X$ is a $\delta-$hyperbolic metric space,  $W_r= \qcg_o \cap \s_r(o)$, $q \in \L$ and $k \in  (\qcg_o \cap B(o,r)) \subset X$, then $[k, \pi_r(q)]\cup [\pi_r(q),q)$ is a $(C,C)-$quasigeodesic and further, $d(\pi_r(q), [k,q) \cap W_r) \leq C$.
\end{cor}

We restate the last statement of Corollary \ref{nppfrominfty} in the form that we shall use, unwinding the definition of the Busemann function $\beta_q(o,k)$ based at $q$:

\begin{cor}\label{betavsspha}
	Given $\delta>0$ there exists $C>0$ such that the following holds for any $r>0$: if $X$ is a $\delta-$hyperbolic metric space,  $W_r= \qcg_o \cap \s_r(o)$, $q \in \L$ and $k \in  \qcg_o \cap B(o,r) \subset X$, then $$|\beta_q(o,k) - (r-d(k, \pi_r(q)))| \leq C.$$
\end{cor}

\noindent {\bf Proof of Proposition \ref{cr}:}
By Corollary \ref{betavsspha} we have,
$$
{\max_{g \in B_r} [exp(v \,  \beta_q \, (o, g.o))]}  \asymp{exp(v \,(r- d_X(\pi_r(q), B_r.o))}
.$$
Since $V_r = exp(vr)$, we have
$$
\frac{1}{V_r}\ {\max_{g \in B_r} [exp(v \,  \beta_q \, (o, g.o))]}
\asymp {exp(v \,(- d_X(\pi_r(q), B_r.o))}
.$$
Hence,
$$C_r 
\asymp  \int_{\L}{exp(v \,(- d_X(\pi_r(q), B_r.o))}d\mu(q)
=  \int_{\L}{\fexr(x)}d\mu(x).$$
This proves the first asymptotic equality of Proposition \ref{cr}.

A standard argument using the Sullivan shadow lemma (see for instance \cite[Proposition 6.1]{coornert-pjm} or \cite[Proposition 3.11]{calegarimaher})  shows that the projection $\pi_r: \L \to \s_r(o)$ and the shadow map from $\s_r(o)$ to $\L$ may be used as approximate inverses of each other for large $r$. Hence, integrals over $\s_r(o)$, equipped with the averaging measure $\mu_r$, converge, up to uniform  multiplicative constants, to the integral  over $(\L,\mu)$. Thus, there exists $R_0 >0$ such that for $r\geq R_0$,
$$C_r \asymp \int_{\s_r(o)}{exp(v \,(- d_X(x, B_r.o))}d\mu_r(x)
=   \int_{\s_r(o)} \fexr  d\mu_r(x),$$
completing the proof of Proposition \ref{cr}.
\hfill $\Box$

\section{Vanishing and non-vanishing ECG}\label{sec-mainthm} 
In this section, we shall prove the main theorems of the paper.
\subsection{Non-vanishing ECG}\label{sec-nvecg}

\begin{theorem}\label{main-nvecg}
	The following triples $(G,X,\L)$ have non-vanishing  extremal cocycle growth:
	\begin{enumerate}
		\item $(X,d)$ is a  proper Gromov-hyperbolic metric space equipped with
		a  properly discontinuous convex cocompact isometric action of a group $G$. The limit set $\L$ of $G$ is equipped with a Patterson-Sullivan measure $\mu$.
		\item $(X,d)$ is a proper complete CAT(-1) space equipped with
		a  properly discontinuous  isometric action of a group $G$ such that $M=X/G$ has non-arithmetic length spectrum. The limit set $\L$ of $G$ is equipped with a Patterson-Sullivan measure $\mu$. Further, assume that for $M=X/G$, the Bowen-Margulis measure $\mubm(UM)$ is finite. 
		\item $(X,d)$ is the Teichm\"uller space $Teich(S)$, $G=MCG(S)$, $\L=\partial X = \PMF(S)$ and $\mu$ is the Thurston conformal density based at  a {\it generic} base-point $o \in X$ (i.e.\ $o$ belongs to a full measure subset of $M=X/G$).
		\item $(X,d)$ is a  symmetric space of non-compact type equipped with
		a  properly discontinuous  isometric action of a lattice $G$.  The limit set $\L$ is  the Furstenberg boundary $\partial_FX$ embedded canonically in $\partial X$ as the $K-$orbit of the barycenter of a Weyl chamber at infinity \cite{albgafa}. The limit set $\L$ of $G$ is equipped with the Patterson-Sullivan measure $\mu$.
	\end{enumerate}
	Hence, in all the above cases, the associated group indexed stationary random fields (via the \rosin representation) $\{Y_g:=Y_g (\L,\mu, \{\phi_g\},  \{c_g\}), g \in G,\}$ is  \iid-like (see Definition \ref{def-iidlike}) with respect to the behavior of partial maxima.
\end{theorem}

\begin{proof} We give a case-by-case argument:\\
	{\bf Item 1:} This will  follow immediately from Proposition \ref{cr} if we can prove that
	$\fexr(x)$ is uniformly bounded below point-wise on $\s_r(o)$ (independent of $r$). Since $\fexr = exp(-vd_X(x,  B_r.o)),$
	the point-wise lower bound on $\fexr$ will follow from a pointwise upper bound on $d_X(x,  B_r.o)$ for $x \in \sro$. But this is an immediate consequence of the fact that $G$ acts on $X$ cocompactly.\\
	{\bf Item 2:} Let $P:(X,o) \to (M,y_0)$ denote the based quotient map. Fix $r > 0$ and let $M_0=\{m \in M\vert d(m,y_0) \leq r\}$. Let $UM_0$ denote the restriction of the bundle $UM$ to $M_0$. Since $\mubm$ is a Borel measure, we can assume that $\mubm (UM_0) >0$. After normalizing $\mubm(UM) = 1$, we therefore assume that $\mubm (UM_0)=\eta >0$.
	
	Let $S_{r,0}= \{ x \in \sro \vert P(x) \in M_0\}$. Also let $\{\mu_r\}$ be the family of averaging measures in Item (1) of \ref{egs}. Equidistribution of the spheres $P(\sro)$ in $M$, with respect to $\{\mu_r\}$   follows from Theorem \ref{em}. Hence
	$\mu_r (S_{r,0}) \to \eta >0$ as $r \to \infty$. Proposition \ref{prop-crucialsuffcond} now gives the result. \\
	{\bf Item 3:} Equidistribution of spheres in the context of $Teich(S)$ is given by Theorem \ref{teichmix}. Thus, the proof of Item (2) goes through mutatis mutandis, using Theorem \ref{teichmix} in place of Theorem \ref{em}.\\
	{\bf Item 4:}  Equidistribution of spheres in the context of symmetric spaces is given by Theorem \ref{hom}. The proof of Item (2) goes through in this case using Theorem \ref{hom} in place of Theorem \ref{em}.
	
	The last statement of Theorem \ref{main-nvecg} now follows from Theorem \ref{main_prob}.
\end{proof}

\begin{rmk}
	An alternate argument for Item 2 above can be given by directly invoking Roblin's Theorem \ref{roblindich} for the asymptotics of $V_G(x,y,n)$ when $\mubm(UM) < \infty$. However
	the proof here  generalizes directly to Items 3, 4.
\end{rmk}

\subsection{Vanishing ECG}\label{sec:ecggiibm} The purpose of this subsection is to prove:

\begin{theorem}\label{ecg-gi}
	Let	$(X,d)$ be a  proper CAT(-1) space equipped with
	a base-point $o$ and	a  properly discontinuous  isometric action of a group $G$. The limit set $\L (\subset \partial X)$ of $G$ is equipped with a Patterson-Sullivan measure $\mu$. Suppose that the associated
	Bowen-Margulis measure $\mubm(UM)$ of the  unit tangent bundle  is infinite. Then 
	the action of $G$ on $(\L,\mu)$ has vanishing extremal cocycle growth.
\end{theorem}
\begin{proof}
	To prove that extremal cocycle growth is zero, it suffices to show  the following. For any $C>0$, let $G_C$ denote the $C-$neighborhood of the $G.o$orbit in $\qcg_o$.  Let $\sro = \Sigma_r(o,X)$ be the boundary of the $r-$ball  $B(o,r)$ about $o$ in $\qcg_o$. Note that $|\sro| \asymp |B_r| \asymp e^{vr} = V_r.$ Let $\partial_r(C)$ denote $G_C\cap \sro$. Define $$m_r(C):= \frac{\mu_r(\partial_r(C))}{\mu_r(\sro)}.$$ By Roblin's Theorem \ref{roblindich},  $\mubm(UM)=\infty$ implies that  for all $C>0$ $\mu_r(\partial_r(C)) = o(V_r)$. Hence $m_r(C) \to 0$ as $r \to \infty$.  From Lemma \ref{discrepancylarge} below, it follows  that ECG vanishes, i.e.\ $\lim_{r \to \infty} C_r = 0.$
\end{proof}

\begin{lemma}\label{discrepancylarge}
	If $m_r(C) \to 0$ as $r \to \infty$, then  $\lim_{r \to \infty} C_r = 0.$
\end{lemma}
\begin{proof} Let $\fexr$ be as in Proposition \ref{cr}. Since $m_r(C) \to 0$ as $r \to \infty$, it follows that for all $\epsilon >0$, there exists $N \in \natls$ such that for all $r \geq N$,
	$m_r(C) <\epsilon$,
	and $\fexr (x) \leq e^{-C} $ for all $x \in (\sro \setminus \partial_r(C))$. Hence by Proposition \ref{cr}, $$C_r \asymp \int_\sro  \fexr (x) d\mu_r(x) \leq e^{-C} + \epsilon.$$
	Since   $\epsilon$ can be made arbitrarily small and $ C$  arbitrarily large,  $\lim_{r \to \infty} C_r=0$.
\end{proof}

\subsection{Normal subgroups of hyperbolic groups}\label{sec:ecgnormal} When $X$ is the Cayley graph of a  free group with respect to a standard set of generators, Item 2 of  Theorem \ref{main-nvecg} does not apply as the geodesic flow is not mixing in this case (mixing fails more generally a hyperbolic group equipped with the word metric \cite{bader-furman}).
We deal in this section with subgroups $H$ of hyperbolic groups $G$, especially
when $H$ is  normal. For the purposes of this subsection,  $X=\Gamma$ will be a Cayley graph of $G$ with respect to a finite set of generators.
It follows immediately from  Theorem \ref{main-nvecg} that if $H$ is a finite index subgroup of $G$, then the action of  $H$  on $(\partial G, \mups)$ has  non-vanishing extremal cocycle growth.
Assume henceforth that  $H$ is an infinite index subgroup of $G$. $H$ is said to be {\bf co-amenable} in $G$ if the left action of $G$ on the (right) coset space $\Gamma/H$ is amenable. We shall use:

\begin{theorem}\cite{cds} \label{cds}
	For $G, H, X (=\Gamma)$ as above, let $v_G$ and $v_H$ denote the exponential growth rates of $G$ and $H$
	acting on $X$. Then  $H$ is co-amenable in $G$ if and only if $v_H=v_G$.
\end{theorem}

Let $V_H(1,1,n)$ be the growth series of $H$ acting on $\Gamma$.
We now observe:
\begin{prop}\label{ecg-hypgp-nonamen}
	Let $G$ be a  hyperbolic group and $H$ a subgroup so that the left $G-$action on the (right) coset space $X/H$ is
	non-amenable (in particular  when $H$ is normal,  the quotient group $G/H$ is non-amenable).  Then ECG for the $H-$action on the boundary $\partial G$ of $G$, equipped with its Patterson-Sullivan measure vanishes.
	If $v_H=v_G$, $M=X/H$  and   the Poincar\'e series for the $H-$action on $\Gamma$ converges at $v_H$, then  ECG for the $H-$action on its limit set $\L_H \subset \partial G$ vanishes.
\end{prop}
\begin{proof}
	It follows immediately from Theorem \ref{cds} that when the left $G-$action on the (right) coset space $G/H$ is
	non-amenable, then  the critical exponent $v_H$ of $H$ is strictly less than  the critical exponent $v_G$ of $G$.  Then $V_H(1,1,n)=o(exp(v_Gn))$.
	The proof of Theorem \ref{ecg-gi} now shows that ECG is vanishing in this case. 
	
	When the Poincar\'e series for the $H-$action on $\Gamma$ converges at $v_H=v_G$, then also, $V_H(1,1,n)=o(exp(v_Gn))$. The proof of Theorem \ref{ecg-gi} again shows that ECG is vanishing in this case. 
\end{proof}

We finally come to:

\begin{theorem}\label{ecg-hypgp-normal}
	Let $G$ be a  hyperbolic group and $H$ an infinite normal subgroup of infinite index. Let $X=\Gamma$ denote a Cayley graph of $G$ with respect to a finite generating set. Let $\mups$ denote the Patterson-Sullivan measure of $H$ on the limit set $\L_H = \partial G$.  Then the ECG for the triple $(X, \L_H,H)$ vanishes.
\end{theorem}

\begin{proof}
	It suffices,  by Proposition \ref{ecg-hypgp-nonamen}, to assume that $v_H=v_G=v$ and that the Poincar\'e series of $H$ diverges at $v$. Further,
	as in the proof of Theorem \ref{ecg-gi}, it is enough to show that 
	$V_H(1,1,m)=o(exp\, (vm))$.

	We now invoke a Theorem due to Matsuzaki,  Yabuki and Jaerisch
	\cite[Theorem 4.2]{matsuzaki}, \cite[Theorem 1.2]{matsuzaki2} that ensures that the Patterson-Sullivan measure $\mups$ of $H$ is, up to uniformly bounded multiplicative constants, invariant under the action of $G$. We normalize so that $\mups$ is constructed with base-point $1\in H$. Thus, for all $g \in G$, $$g^\ast \mups \asymp \mups_{g^{-1}},$$ where the suffix ${g^{-1}}$ indicates the shifted base-point. Hence, for all
	$q \in \L_H = \partial G$, $$g^\ast \mups (q) + (g^{-1})^\ast \mups (q) \asymp \mups(q), $$ where $\asymp$ indicates uniform multiplicative constants independent of $g$. 
	Since $G/H$ is infinite, we can choose distinct $g_1, g_1^{-1} \cdots , g_n, g_n^{-1} $ such that for all
	$q \in \L_H = \partial G$, $$\sum_1^n [g_i^\ast \mups (q) + (g_i^{-1})^\ast  \mups (q) ] \asymp n \mups(q).$$
	
	As usual, let $B(1,m)$ denote the $m-$ball in $X=\Gamma$.
	It follows that for any  distinct $g_1, g_1^{-1} \cdots, g_n , g_n^{-1} $, there exists $N$ such that for  $m \geq N$, 
	$$\sum_1^n (V_H(1,g_i,m) + V_H(1,g_i^{-1},m)) \asymp n \, V_H(1,1,m).$$
	Since $$\sum_1^n (V_H(1,g_i,m) + V_H(1,g_i^{-1},m)) \leq V_G(1,1,m)
	\asymp exp (vm),$$ it follows that for  $m\geq N$, 
	$$V_H(1,1,m) \lesssim \frac{1}{n} \, exp (vm).$$ Since $n$ can be made arbitrarily large, $V_H(1,1,m)=o(exp\, (vm))$ as required.
\end{proof}

\medskip

\noindent {\bf Concluding Remarks:}\\ 

\smallskip

\noindent (1) Replacing Albuquerque's results \cite{albgafa} in Section \ref{sec:psmss} by a Theorem of Link \cite[Theorem A]{link-gd} gives immediately an analog of Theorem \ref{main-nvecg} Item (4) for lattices in products
of negatively curved manifolds.
\\

\noindent (2) The proof of Item 1 of
Theorem \ref{main-nvecg} goes through without modification when $G$ acts cocompactly on $(X,d)$ when the latter is only quasi-ruled in the sense of
\cite[Section 1.7]{bhm} instead of being a geodesic metric space. Thus, let $X = \Gamma (G,S)$ where $G$ is hyperbolic and $S$ is a finite generating set.
Let $\mu$ be a finitely supported symmetric measure on $G$ whose support generates $G$. Let $\nu$ be the hitting measure on $\partial G$. Let $d$ be the {\bf Green metric}  on  $X$ \cite{bhm} and let $\Lambda = (\partial G, \nu)$. Then proof of Item 1 of
Theorem \ref{main-nvecg} goes through and shows that the action of $G$ on $\Lambda$ has non-vanishing ECG. Hence by Theorem \ref{main_prob}, the behavior of partial maxima is \iid-like. Note that in this case, the Busemann function is computed with respect to the Green metric rather than the word metric \cite{bhm}.\\

\noindent (3) An exact analog of Ricks' theorem \cite{ricks} on mixing and convergence of spherical averages \ref{em} is absent at this point for general Gromov-hyperbolic spaces. This is the only obstruction in obtaining an exact analog of Theorem \ref{main-nvecg} Item (2) for general Gromov-hyperbolic spaces.\\

\noindent (4) {\bf Dependence on $\alpha$:} For a stationary $\sas$ random field $Y_g = Y_g(S, \mu, \phi_g,c_g,f)$, ECG (Definition \ref{def-ecgps}) identifies the qualitative behavior of partial maxima when $\mu$ is a probability measure and $f$ is a constant function. Note that this qualitative behavior (of being \iid-like) is independent of $\alpha$. Thus, to determine the dependence of partial maxima on $\alpha$, we really need to investigate the general case of non-constant $f$.\\

\noindent (5) It might be worthwhile to extract  axiomatically the essential features  from all the examples  of non-vanishing ECG in Theorem \ref{main-nvecg} to provide a general sufficient condition.\\

\noindent (6) It was kindly pointed out to us by the referee that the main result of the paper \cite{cdst} by Coulon-Dougall-Schapira-Tapie generalizes Theorem \ref{cds} as follows. In Theorem \ref{cds}, $X$ is assumed to be a Cayley graph of $G$. However, \cite{cdst} allows Theorem \ref{cds} to go through when
	$G$ acts cocompactly on $X$, or more generally when $X/G$ has finite Bowen-Margulis measure. This allows all the results for subgroups of hyperbolic groups  to go through in this more general context.\\

We hope to take up some of the unexplored issues in the above list in subsequent work.\\

\noindent{\bf Acknowledgments:} The  authors would like to thank Uri Bader,  Remi Coulon, Alex Furman, Anish Ghosh, Amos Nevo, Michah Sageev, Gennady Samorodnitsky and Sourav Sarkar for extremely helpful conversations, and   Apoorva Khare, Soumik Pal and D. Yogeshwaran for their comments on an earlier version. Parts of this work were completed during visits to International Centre for Theoretical Sciences, Bengaluru (Discussion Meeting on Surface Group Representations and Projective Structures, December 2018), Technion, Israel Institute of Technology (May 2018),   Fields Institute, University of Toronto (October, 2018),  Tata Institute of Fundamental Research, Mumbai (April 2018), and the Mathematical Sciences Research Institute (March 2022). We thank these institutions for their hospitality. Finally, we thank the anonymous referee for helpful comments.


\begin{thebibliography}{ABEM12}
	
	\bibitem[ABEM12]{abem}
	Jayadev Athreya, Alexander Bufetov, Alex Eskin, and Maryam Mirzakhani.
	\newblock Lattice point asymptotics and volume growth on {T}eichm\"{u}ller
	space.
	\newblock {\em Duke Math. J.}, 161(6):1055--1111, 2012.
	
	\bibitem[Alb99]{albgafa}
	P.~Albuquerque.
	\newblock {Patterson-Sullivan theory in higher rank symmetric spaces}.
	\newblock {\em Geom. funct. anal., 9(1)}, pages 1--28, 1999.
	
	\bibitem[BAPP16]{pp-eq}
	A.~Broise-Alamichel, J.~Parkkonen, and F.~Paulin.
	\newblock Equidistribution and counting under equilibrium states in negatively
	curved spaces and graphs of groups. applications to non-archimedean
	diophantine approximation.
	\newblock {\em preprint, arXiv:1612.06717}, 2016.
	
	\bibitem[BF17]{bader-furman}
	U.~Bader and A.~Furman.
	\newblock {Some ergodic properties of metrics on hyperbolic groups}.
	\newblock {\em preprint, arXiv:1707.02020}, page 23pp., 2017.
	
	\bibitem[BHM11]{bhm}
	S.~Blachere, P.~Haissinsky, and P.~Mathieu.
	\newblock Harmonic measures versus quasiconformal measures for hyperbolic
	groups.
	\newblock {\em Ann. Sci. \'Ec. Norm. Sup\'er. (4) 44, no. 4}, pages 683--721,
	2011.
	
	\bibitem[Cal13]{calegari-notes}
	Danny Calegari.
	\newblock The ergodic theory of hyperbolic groups.
	\newblock In {\em Geometry and topology down under}, volume 597 of {\em
		Contemp. Math.}, pages 15--52. Amer. Math. Soc., Providence, RI, 2013.
	
	\bibitem[CDS18]{cds}
	R\'{e}mi Coulon, Fran\c{c}oise Dal'Bo, and Andrea Sambusetti.
	\newblock Growth gap in hyperbolic groups and amenability.
	\newblock {\em Geom. Funct. Anal.}, 28(5):1260--1320, 2018.
	
	\bibitem[CDST18]{cdst}
	R\'{e}mi Coulon,  Rhiannon Dougall, Barbara Schapira, and Samuel Tapie.
	\newblock Twisted Patterson-Sullivan measures and applications to amenability and coverings.
	\newblock {\em preprint}, arXiv:1809.10881,
	 2018.
	
	\bibitem[CM07]{cm-gafa}
	Chris Connell and Roman Muchnik.
	\newblock Harmonicity of quasiconformal measures and {P}oisson boundaries of
	hyperbolic spaces.
	\newblock {\em Geom. Funct. Anal.}, 17(3):707--769, 2007.
	
	\bibitem[CM15]{calegarimaher}
	Danny Calegari and Joseph Maher.
	\newblock Statistics and compression of scl.
	\newblock {\em Ergodic Theory Dynam. Systems}, 35(1):64--110, 2015.
	
	\bibitem[Coo93]{coornert-pjm}
	M.~Coornaert.
	\newblock {Mesures de Patterson-Sullivan sur le bord d\'un espace hyperbolique
		au sens de Gromov}.
	\newblock {\em Pacific J. Math. 159, no. 2}, pages 241--270, 1993.
	
	\bibitem[EM93]{eskin-mcmullen}
	A.~Eskin and C.T. McMullen.
	\newblock Mixing, counting, and equidistribution in lie groups.
	\newblock {\em Duke Math. J. 71, no. 1}, pages 181--209, 1993.
	
	\bibitem[Fur02]{furman-coarse}
	Alex Furman.
	\newblock Coarse-geometric perspective on negatively curved manifolds and
	groups.
	\newblock In {\em Rigidity in dynamics and geometry ({C}ambridge, 2000)}, pages
	149--166. Springer, Berlin, 2002.
	
	\bibitem[Gro81]{gromov-hypm}
	M.~Gromov.
	\newblock Hyperbolic manifolds, groups and actions.
	\newblock In {\em Riemann surfaces and related topics: {P}roceedings of the
		1978 {S}tony {B}rook {C}onference ({S}tate {U}niv. {N}ew {Y}ork, {S}tony
		{B}rook, {N}.{Y}., 1978)}, volume~97 of {\em Ann. of Math. Stud.}, pages
	183--213. Princeton Univ. Press, Princeton, N.J., 1981.
	
	\bibitem[KM96]{KM2}
	D.~Y. Kleinbock and G.~A. Margulis.
	\newblock Bounded orbits of nonquasiunipotent flows on homogeneous spaces.
	\newblock In {\em Sina\u{\i}'s {M}oscow {S}eminar on {D}ynamical {S}ystems},
	volume 171 of {\em Amer. Math. Soc. Transl. Ser. 2}, pages 141--172. Amer.
	Math. Soc., Providence, RI, 1996.
	
	\bibitem[KM99]{KM1}
	D.~Y. Kleinbock and G.~A. Margulis.
	\newblock Logarithm laws for flows on homogeneous spaces.
	\newblock {\em Invent. Math.}, 138(3):451--494, 1999.
	
	\bibitem[Lin15]{link-gd}
	Gabriele Link.
	\newblock Generalized conformal densities for higher products of rank one
	{H}adamard spaces.
	\newblock {\em Geom. Dedicata}, 178:351--387, 2015.
	
	\bibitem[Lin18]{glink18}
	Gabriele Link.
	\newblock Equidistribution and counting of orbit points for discrete rank one
	isometry groups of hadamard spaces.
	\newblock {\em preprint, arXiv:1808.03223}, 2018.
	
	\bibitem[Mas82]{Masur}
	Howard Masur.
	\newblock Interval exchange transformations and measured foliations.
	\newblock {\em Ann. of Math. (2)}, 115(1):169--200, 1982.
	
	\bibitem[Mas93]{masurloglaw}
	Howard Masur.
	\newblock Logarithmic law for geodesics in moduli space.
	\newblock {\em Mapping class groups and moduli spaces of {R}iemann surfaces
		({G}\"{o}ttingen, 1991/{S}eattle, {WA}, 1991)}, 150:229--245, 1993.
	
	\bibitem[Mit98a]{mitra-ct}
	M.~Mitra.
	\newblock Cannon-{T}hurston {M}aps for {H}yperbolic {G}roup {E}xtensions.
	\newblock {\em Topology 37}, pages 527--538, 1998.
	
	\bibitem[Mit98b]{mitra-trees}
	M.~Mitra.
	\newblock Cannon-{T}hurston {M}aps for {T}rees of {H}yperbolic {M}etric
	{S}paces.
	\newblock {\em Jour. Diff. Geom.48}, pages 135--164, 1998.
	
	\bibitem[Miy14]{miyachi-hb}
	Hideki Miyachi.
	\newblock Extremal length geometry.
	\newblock In {\em Handbook of {T}eichm\"{u}ller theory. {V}ol. {IV}}, volume~19
	of {\em IRMA Lect. Math. Theor. Phys.}, pages 197--234. Eur. Math. Soc.,
	Z\"{u}rich, 2014.
	
	\bibitem[MS91]{MasurSmillie}
	Howard Masur and John Smillie.
	\newblock Hausdorff dimension of sets of nonergodic measured foliations.
	\newblock {\em Ann. of Math. (2)}, 134(3):455--543, 1991.
	
	\bibitem[MY09]{matsuzaki}
	K.~Matsuzaki and Y.~Yabuki.
	\newblock {The Patterson-Sullivan measure and proper conjugation for Kleinian
		groups of divergence type}.
	\newblock {\em Ergod. Th. \& Dynam. Sys., 29}, pages 657--665, 2009.
	
	\bibitem[MYJ15]{matsuzaki2}
	K.~Matsuzaki, Y.~Yabuki, and J.~Jaerisch.
	\newblock {Normalizer, divergence type and Patterson measure for discrete
		groups of the Gromov hyperbolic space}.
	\newblock {\em preprint, arXiv:1511.02664}, 2015.
	
	\bibitem[Nev17]{Nevo}
	Amos Nevo.
	\newblock Equidistribution in measure-preserving actions of semisimple groups :
	case of sl(2;r).
	\newblock {\em preprint, arxiv:1708.03886}, 2017.
	
	\bibitem[OS12]{os-inv}
	Hee Oh and Nimish Shah.
	\newblock The asymptotic distribution of circles in the orbits of {K}leinian
	groups.
	\newblock {\em Invent. Math.}, 187(1):1--35, 2012.
	
	\bibitem[OS13]{os-jams}
	Hee Oh and Nimish~A. Shah.
	\newblock Equidistribution and counting for orbits of geometrically finite
	hyperbolic groups.
	\newblock {\em J. Amer. Math. Soc.}, 26(2):511--562, 2013.
	
	\bibitem[Pat76]{patterson-acta}
	S.~J. Patterson.
	\newblock The limit set of a fuchsian group.
	\newblock {\em Acta Math., 136}, pages 241--273, 1976.
	
	\bibitem[PP14]{pp-skin}
	Jouni Parkkonen and Fr\'ed\'eric Paulin.
	\newblock Skinning measures in negative curvature and equidistribution of
	equidistant submanifolds.
	\newblock {\em Ergodic Theory Dynam. Systems}, 34(4):1310--1342, 2014.
	
	\bibitem[Ric17]{ricks}
	R.~Ricks.
	\newblock {Flat strips, Bowen-Margulis measures, and mixing of the geodesic
		flow for rank one ${\rm CAT}(0)$ spaces}.
	\newblock {\em Ergodic Theory Dynam. Systems 37, no. 3}, pages 939--970, 2017.
	
	\bibitem[Rob03]{roblin-memo}
	T.~Roblin.
	\newblock Ergodicite et equidistribution en courbure negative.
	\newblock {\em Mem. Soc. Math. Fr. (N.S.) No. 95}, pages vi+96, 2003.
	
	\bibitem[Ros94]{Rosinski:1994}
	Jan Rosi{\'n}ski.
	\newblock On uniqueness of the spectral representation of stable processes.
	\newblock {\em J. Theoret. Probab.}, 7(3):615--634, 1994.
	
	\bibitem[Ros95]{Rosinski:1995}
	Jan Rosi{\'n}ski.
	\newblock On the structure of stationary stable processes.
	\newblock {\em Ann. Probab.}, 23(3):1163--1187, 1995.
	
	\bibitem[Ros00]{Rosinski:2000}
	Jan Rosi{\'n}ski.
	\newblock Decomposition of stationary {$\alpha$}-stable random fields.
	\newblock {\em Ann. Probab.}, 28(4):1797--1813, 2000.
	
	\bibitem[Roy17]{roy:2017}
	Parthanil Roy.
	\newblock Maxima of stable random fields, nonsingular actions and finitely
	generated abelian groups: A survey.
	\newblock {\em Indian Journal of Pure and Applied Mathematics}, 48(4):513--540,
	2017.
	
	\bibitem[RS08]{roy:samorodnitsky:2008}
	Parthanil Roy and Gennady Samorodnitsky.
	\newblock Stationary symmetric {$\alpha$}-stable discrete parameter random
	fields.
	\newblock {\em J. Theoret. Probab.}, 21(1):212--233, 2008.
	
	\bibitem[Sam04]{samorodnitsky:2004a}
	G.~Samorodnitsky.
	\newblock Extreme value theory, ergodic theory, and the boundary between short
	memory and long memory for stationary stable processes.
	\newblock {\em Ann. Probab.}, 32:1438--1468, 2004.
	
	\bibitem[SR18]{sarkar:roy:2016}
	Sourav Sarkar and Parthanil Roy.
	\newblock Stable random fields indexed by finitely generated free groups.
	\newblock {\em Annals of Probability}, 46(5):2680 -- 2714, 2018.
	
	\bibitem[ST94]{samorodnitsky:taqqu:1994}
	Gennady Samorodnitsky and Murad~S. Taqqu.
	\newblock {\em Stable non-{G}aussian random processes}.
	\newblock Stochastic Modeling. Chapman \& Hall, New York, 1994.
	\newblock Stochastic models with infinite variance.
	
	\bibitem[Sul79]{sullivan-pihes}
	D.~Sullivan.
	\newblock The density at infinity of a discrete group of hyperbolic motions.
	\newblock {\em Publ. Math. I.H.E.S., 50}, pages 171--202, 1979.
	
	\bibitem[Sul84]{sullivan-acta84}
	D.~Sullivan.
	\newblock Entropy, hausdorff measures old and new and limit sets of
	geometrically finite kleinian groups.
	\newblock {\em Acta Math., 153}, pages 259--277, 1984.
	
	\bibitem[Thu80]{thurstonnotes}
	W.~P. Thurston.
	\newblock The {G}eometry and {T}opology of 3-{M}anifolds.
	\newblock {\em Princeton University Notes}, 1980.
	
	\bibitem[Vee82]{Veech}
	William~A. Veech.
	\newblock Gauss measures for transformations on the space of interval exchange
	maps.
	\newblock {\em Ann. of Math. (2)}, 115(1):201--242, 1982.
	
\end{thebibliography}
\end{document}